\documentclass[11pt]{amsart}

\usepackage{soul}
\usepackage{amsfonts}
\RequirePackage{natbib}
\usepackage{graphicx}
\usepackage{color}
\usepackage{amssymb}
\usepackage{amsfonts}
\usepackage{amsthm}
\usepackage[mathcal]{eucal}
\usepackage{ipa}

\usepackage[justification=centering,font=footnotesize]{caption}
\usepackage[labelformat=simple]{subfig}

\newtheorem{theoreme}{Theorem}[section]
\newtheorem{lemma}[theoreme]{Lemma}
\newtheorem{proposition}[theoreme]{Proposition}

\newtheorem{definition}[theoreme]{Definition}
\newtheorem{remark}[theoreme]{Remark}
\newtheorem{corollary}[theoreme]{Corollary}

\newtheorem{example}[theoreme]{Example}

\def\R{\mathbb R}
\def\N{\mathbb N}
\def\Z{\mathbb Z}

\newcommand{\red}[1]{\textcolor{black}{#1}}
\newcommand{\blue}[1]{\textcolor{black}{#1}}

\newcommand{\begitem}{\begin{itemize}}
\newcommand{\finit}{\end{itemize}}

\newcommand{\tr}{{\mbox{\tiny \sf T}}}

\newcommand\restr[2]{{
  \left.\kern-\nulldelimiterspace 
  #1 
  \vphantom{\big|} 
  \right|_{#2} 
  }}

\newcommand{\bbx}{\overline{\mathbf{x}}}

\newcommand{\diag}{\mathrm{diag}}
\newcommand{\bS}{S}

\newcommand{\bzeta}{\boldsymbol{\zeta}}

\newcommand{\PPinv}{\Inv(\theta)}
\newcommand{\PP}{\mathcal{P}}

\newcommand{\eps}{\varepsilon}
\newcommand{\vphi}{\varphi}
\newcommand{\rr}{r}

\DeclareMathOperator{\support}{supp}
\DeclareMathOperator{\interior}{int}

\DeclareMathOperator{\Inv}{Inv}

\DeclareMathOperator{\argmax}{argmax}

\begin{document}

\title{Robust permanence for ecological maps }
\author[G. Roth]{Gregory Roth}
\email{greg.roth51283@gmail.com}
\address{Institute for Biodiversity and Ecosystem Dynamics (IBED), University of Amsterdam, The Netherlands.}
\author[P.L. Salceanu]{Paul L. Salceanu}
\email{salceanu@louisiana.edu}
\address{Department of Mathematics, University of Louisiana, Lafayette, LA 70504 USA}
\author[S.J. Schreiber]{Sebastian J. Schreiber}
\email{sschreiber@ucdavis.edu}
\address{Department of Evolution and Ecology, One Shields Avenue, University of California, Davis, CA 95616 USA}

\begin{abstract}
We consider ecological difference equations of the form $X_{t+1}^i =X_t^i A_i(X_t)$ where $X_t^i$ is a vector of densities corresponding to the subpopulations of species $i$ (e.g. subpopulations of different ages or living in different patches), $X_t=(X_t^1,X_t^2,\dots,X_t^m)$ is state of the entire community, and $A_i(X_t)$ are matrices determining the update rule for species $i$. These equations are permanent if they are dissipative and the extinction set $\{X: \prod_i \|X^i\|=0\}$ is repelling. If permanence persists under perturbations of the matrices $A_i(X)$, the equations are robustly permanent. We provide sufficient and necessary conditions for robust permanence in terms of Lyapunov exponents for invariant measures supported by the extinction set.  Applications to ecological and epidemiological models are given. 
\end{abstract}

\maketitle
\section{Introduction}

A fundamental question in ecology is to understand under what minimal conditions a community of species persist in the long run. Historically, theoretical ecologists characterize persistence by the existence of an asymptotic equilibrium in which the proportion of each population is strictly positive~\citep{may-75,roughgarden-79}. More recently, coexistence was equated  with the existence of an attractor bounded away from extinction~\citep{hastings-88}, a definition that ensures populations will persist despite small, random \red{environmental} perturbations~\citep{jtb-06,dcds-07}. However, ``environmental perturbations are often vigourous shake-ups, rather than gentle stirrings''~\citep{jansen-sigmund-98}. To account for large, but rare, perturbations,  the concept of permanence, or uniform persistence, was introduced in late 1970s~\citep{freedman-waltman-77,schuster-etal-79}. Uniform persistence requires that asymptotically species densities remain uniformly bounded away from extinction. In addition, permanence requires that the system is dissipative i.e. asymptotically species densities remain uniformly bounded from above.

Various mathematical approaches exist for verifying permanence~\citep{hutson-schmitt-92,smith-thieme-11} including topological characterizations with respect to chain recurrence~\citep{butler-waltman-86,hofbauer-so-89}, average Lyapunov functions~\citep{hofbauer-81,hutson-84,garay-hofbauer-03}, and measure theoretic approaches~\citep{schreiber-00,hofbauer-schreiber-10}. The latter two approaches involve the long-term per-capita growth rates of species when rare. For continuous-time, unstructured models of the form
\begin{equation}\label{eq:Kolmogorov}
\frac{dx^i}{dt}=f_i(x)x^i, \: i=1,...,m,
\end{equation}
where $x=(x^1,\dots,x^m)$ is the vector of population densities, the long-term growth rate of species $i$ with initial community state $x$ equals \[
r_i(x)=\limsup_{t \to \infty}\frac{1}{t}\int_0^t f_i(x_s)ds\]
where $x_s$ denotes the \red{solution} of \eqref{eq:Kolmogorov} with initial condition $x_0=x$.
\cite{hofbauer-81} showed, under appropriate assumptions, that the system is permanent provided
\begin{description}
\item[Condition $\bold{C}$] \emph{there exist positive weights $p_1,\dots,p_m$ associated with each species such that $\sum_i p_i r_i (x)>0$ for any initial condition $x$ with one or more missing species (i.e. $\prod_i x^i = 0$)}.
\end{description}
Intuitively, the community persists if on average the community increases when rare.

Any sensible definition that characterizes the long-term behavior of a population dynamics should be robust under small perturbation of the governing equations. \red{This concept is practical from a modeling standpoint as most population dynamics
models ignore weak interactions between populations. For instance,
the modeler assumes that $f_i$ is independent of $x_j$ where
$j\neq i$ when in fact there is a weak dependence on $x_j$. Hence
it is desirable to know whether ``nearby'' models that include
these interactions as well as the original model  are permanent.} To address this need, the model \eqref{eq:Kolmogorov} is \emph{robust permanent} if it is permanent under small perturbations of the maps $f_i$ (see \citep{hofbauer-sigmund-98,hutson-schmitt-92}). \cite{schreiber-00} and \cite{garay-hofbauer-03} consider the average long-term growth rates of species $i$
\[
r_i(\mu)=\int r_i(x)d\mu
\]
with respect to an ergodic measure $\mu$, and showed that the system \eqref{eq:Kolmogorov} is robust permanent provided
\begin{description}
\item[Condition $\bold{C'}$] \emph{there exist positive weights $p_1,\dots,p_n$ associated with each species such that $\sum_i p_i r_i (\mu)>0$ for any ergodic measure $\mu$ supported by the boundary of the non-negative cone $\R_+^n$}.
 \end{description}
 In fact, conditions $\bold{C}$ and $\bold{C'}$ are equivalent. Moreover, it is sufficient to check  $\bold{C'}$ for each component of a Morse decomposition on the boundary (see  Section 3.2). While \citet{garay-hofbauer-03} extended these results to discrete-time, unstructured models of the form $x_{t+1}^i= f_i(x_t)x_t^i$, their results are restricted to homeomorphisms which excludes many classical models from population biology (e.g. see Section \ref{section:application}).

As populations often exhibit structure (e.g., individuals living in different spatial locations, individuals being of different ages or in different stages of development),  \cite{hofbauer-schreiber-10} obtained robust permanence results for systems of the form
\begin{equation}\label{eq:Structured_Systems_Cont}
\frac{dx^i}{dt}=A_i(x)x^i, \: i=1,...,m,
\end{equation}
where $x^i \red{=(x^{i1},x^{i2},\dots,x^{in_i})^\tr}\in \mathbb{R}^{n_i}$ is the transpose of the row vector of populations abundances of individuals in different states for population $i$, $x = (x^1 , . . . , x^m )$, $A_i(x)$ are non-negative matrices, and $n_1+...+n_m=n$.  For these models, the average long-term growth rates with respect to an ergodic measure $\mu$ are given by \[\red{r_i}(\mu)=\int \langle A_i(x)u_i(x),u_i(x)\rangle\: \mu (dx)\] where $\langle \cdot, \cdot \rangle$ represents the inner product and where, care of a result of  \citet{ruelle-79b}, $u_i(x)$ spans a one dimensional vector space that is invariant under the linear transformation represented by the fundamental solution matrix for $\frac{dy}{dt}=A_i(x_t)y$ with $x_0=x$. With this notion of the long-term growth rate, the necessary condition $\bold{C'}$ for robust permanence  extends to these structured models.

In this paper, we extend conditions ($\bold{C}$) and ($\bold{C'}$) to the discrete-time analogous of \eqref{eq:Structured_Systems_Cont} of the form
\begin{equation}\label{eq:Structured_Systems_Discrete}
X_{t+1}^i=X_t^iA_i(X_t), \: i=1,...,m,
\end{equation}
where $X^i\in \mathbb{R}^{n_i}$ is the row vector of populations abundances of individuals in different states for population $i$ at time $t \in \N$, and $X_t = (X^1_t , . . . , X^m_t )$. Then the long term growth rate $r_i(x)$ of species $i$ corresponds to the dominant Lyapunov exponent associated with the matrix products along the population trajectory:
\begin{equation*}\label{eq:Lyapunov_Exponents}
r_i(x)=\limsup_{t\to \infty}\frac{1}{t}\ln ||A_i(X_0)\cdot \ldots \cdot A_i(X_{t-1})||, \: X_0=x,
\end{equation*}
and the average long-term growth rate with respect to an ergodic measure $\mu$ is
\begin{equation*}\label{eq:Invasion_Rates}
r_i(\mu)=\int r_i(x)\: \mu(dx).
\end{equation*}
This extension to discrete time is non-trivial as the semiflow generated by a discrete dynamical system is not necessarily a homeomorphism, a key ingredient used in the proofs for the continuous-time cases.

Our main result implies that the ``community increases when rare" criterion for robust permanence also applies for the system \eqref{eq:Structured_Systems_Discrete}. More precisely, we show that conditions ($\bold{C}$) and ($\bold{C'}$) are equivalent and imply robust permanence of the system \eqref{eq:Structured_Systems_Discrete}. This result extends (in the case when $n_i=1$ for all $i=1,...,m$) the results in \cite{garay-hofbauer-03} from homeomorphisms to non-invertible maps and the result in \cite{hofbauer-schreiber-10} from continuous-time models to discrete-time models. Our model, assumptions, and definitions of permanence and robust permanence are presented in Section 2. Long-term growth rates for these models and our main theorem are stated in Section 3. A refinement of our result involving Morse decompositions of boundary dynamics is also presented in Section 3. Proofs of most results are presented in Section 4. We apply our results to a series of models from population biology in Section 5. We obtain, for the first time, robust permanence results for the discrete-time Lotka-Volterra equations introduced by \cite{hofbauer-etal-87}. For spatially structured versions of these Lotka-Voltera models, we provide, using a perturbation result, a simple condition for robust permanence. Notably, our condition holds despite the dynamics of each competitor potentially being chaotic. Finally, we obtain a robust permanence condition for a classical epidemiological SIR model.

\section{Model and assumptions}

We study the dynamics of $m$ interacting populations in a constant environment. Each individual in population $i$ can be in one of $n_i$ individual states such as their age, size, or location.  Let $X_t^i = (X_t^{i1}, \dots, X_t^{in_i})$ denote the row vector of populations abundances of individuals in different states for population $i$ at time $t\in \N$. $X_t^i$ lies in the non-negative cone $\R^{n_i}_+$. The \emph{population state} is the row vector  $X_t=(X_t^1, \dots,  X_t^m)$ that lies in the non-negative cone $\R^n_+$ where $n=\sum_{i=1}^mn_i$.

To define the population dynamics, we consider projection matrices for each population that depend on the population state. More precisely, given $X$ the population state, for each $i$, let $A_i(X)$ be a non-negative, $n_i\times n_i$ matrix whose $j$--$k$-th entry corresponds to the contribution of individuals in state $j$ to individuals in state $k$ e.g. individuals transitioning from state $j$ to state $k$ or the mean number of offspring in state $k$ produced by individuals in state $j$.  Using these projection matrices the population dynamic of population $i$ is given by
\begin{equation}\label{Model_Componentwise}
	X_{t+1}^i = X_t^iA_i(X_t).
\end{equation}
where $X_t^i$ multiplies on the left hand side of $A_i(X_t)$ as it is a row vector.  If we define $A(X)$ to be the $n\times n$ block diagonal matrix $\mathrm{diag}(A_1(X),\dots,A_m(X))$, then the dynamics of the interacting populations are given by
\begin{equation}\label{DYN1}
X_{t+1} = X_t A(X_t).
\end{equation}
For these dynamics, we make the following assumptions:
\begin{enumerate}
\item[\textbf{H1:}] For each $i$, $X \mapsto A_i(X)$ is a continuous map into the space of $n_i\times n_i$ non-negative matrices.
\item[\textbf{H2:}] For each population $i$, the matrix $A_i$ has fixed sign structure corresponding to a primitive matrix. More precisely, for each $i$, there is a $n_i\times n_i$, non-negative, primitive matrix $P_i$ such that  the $j$-$k$th entry of $A_i(X)$ equals zero if and only if $j$-$k$th entry $P_i$ equals zero for all $1\le j,k\le n_i$ and $X\in \R^n_+$.
\item[\textbf{H3:}] There exists a compact set $S_A\subset \R^n_+$ such that for all $X_0\in \R^n_+$,  $X_t\in S_A$ for all $t$ sufficiently large. 
\end{enumerate}

Our analysis focuses on whether the interacting populations tend to be bounded away from extinction. Extinction of one or more population corresponds to the population state lying in the \emph{extinction set}
\[
\bS_0 = \{ x \in \R^n_+ : \prod_i \|x^i\|=0\}
\]
where $\|x^i\|=\sum_{j=1}^{n_i} x^{ij}$ corresponds to the $\ell^1$--norm of $x^i$.
\red{For any $\eta\ge 0$,} define
\[
S_\eta=\{ x \in \R^n_+ : \|x^i\| \le \eta \mbox{ for some } i\}.
\]

\begin{definition}\label{defn}
Model \eqref{DYN1} is \emph{permanent} if there exists $\eta>0$ such that for all $x \in \{y \in \R^n_+ : \prod_i \Vert y^i\Vert \neq 0 \}$, there exists $t_0=t_0(x) >0$ such that $X_t \in \R^n_+ \backslash S_{\eta}$ for all $t\ge t_0$, whenever $X_0=x$.
\end{definition}

For $B\subset \R^n_+$ and $\delta>0$, let $N_{\delta}(B)$ be the $\delta$-neighborhood of $B$, i.e.
\[
N_{\delta}(B)=\{x\in \R^n_+ \ : \ \Vert x-y\Vert < \delta \text{ for some } y \in B\}.
\]

We define a $\delta$\emph{-perturbation of model (\ref{DYN1})} to be a system of the form
\[
X_{t+1}=X_tA^{\delta}(X_t),
\]
with $A^{\delta}=\diag(A_1^\delta,\dots,A_m^{\delta})$, that satisfies (i) assumptions \textbf{H1}-\textbf{H3}, (ii) $S_{A^{\delta}} \subset N_{\delta}(S_A)$, and (iii) $\Vert A(x)-A^\delta(x)\Vert \le \delta$ for all $x\in N_{1}(S_A)$.

\begin{definition}
Model (\ref{DYN1}) is \emph{robustly permanent} if there exist $\delta>0$  such that all $\delta$-perturbations of model (\ref{DYN1}) are permanent \red{with a common $\eta>0$ value} i.e. there is a common/uniform region of repulsion around the boundary.
\end{definition}

\section{Results}

\subsection{Long-term growth rates and \red{robustly unsaturated sets}}

Understanding persistence often involves understanding what happens to each population when it is rare. To this end, we need to understand the propensity of the population to increase or decrease in the long term. Since
\[
X_{t}^i= X_0^iA_i(X_0)A_i(X_1)\dots A_i( X_{t-1}),
\]
one might be interested in the long-term ``growth'' of product of matrices
\begin{equation}\label{matrixproducts}
A_i(X_0)A_i(X_1)\dots A_i(X_{t-1})
\end{equation}
as $t\to\infty$. One measurement of this long-term growth rate when $X_0=x$ is
\[
\rr_i(x) = \limsup_{t\to\infty} \frac{1}{t} \log \| A_i(X_0)A_i(X_1)\dots A_i(X_{t-1})\|.
\]
Population $i$ is tending to show periods of increase when $\rr_i(x)>0$ and periods of asymptotic decrease when $\rr_i(x)<0$.

An expected, yet useful property of $\rr_i(x)$ is that $\rr_i(x)\le 0$ whenever $\|x^i\|>0$. In words, whenever population $i$ is present, its per-capita growth rate in the long-term is non-positive. This fact follows from \red{\textbf{H3}}. Furthermore, if $\limsup_{t\to\infty} \|X^i_t\|>0$, we get that $\rr_i(x)=0$. In words, if population $i$'s density infinitely often is bounded below by some minimal density, then its long-term growth rate is zero as it is not tending to extinction and its densities are bounded from above. Both of these facts are consequences of Proposition\red{s} \ref{prop:equality-lyap}, \ref{cor_lambdanegative} and \ref{prop_lambdanul}.

Define the map
\begin{equation}\label{def:phiA}
\begin{array}{ccl}
 \Phi_A : &\R^n_+   &\rightarrow  \R^n_+  \\
  & x & \mapsto  xA(x)
\end{array}
\end{equation}
Let $\Phi^t_A$ denote the composition of $\Phi_A$ with itself $t$ times, for $t \in \N$. \blue{Define the \emph{global attractor} for $\Phi_A$ by $G_A=\cap_{s\ge 0} \overline{\cup_{t\ge s} \Phi_A^t(S_A)}$ where $\overline{B}$ denotes the closure of the set $B$.}
Recall, a Borel probability measure $\mu$ on $\R^n_+$ is \emph{$\Phi_A$-invariant measure}  provided that
\[
\mu(B)=\mu(\Phi_A^{-1}(B))
\]
for all Borel sets $B \subset \R^n_+$ When an invariant measure $\mu$ is statistically indecomposable, it is \emph{ergodic}. More precisely, $\mu$ is ergodic if it can not be written as a convex combination of two distinct invariant measures, i.e. if there exist $0<\alpha<1$ and two invariant measures $\mu_1,\mu_2$ such that $\mu=\alpha \mu_1 +(1-\alpha) \mu_2$, then $\mu_1=\mu_2=\mu$. If $\mu$ is a $\Phi_A$-invariant measure, the subadditive ergodic theorem implies that  \[\rr_i(x) = \lim_{t\to\infty} \frac{1}{t} \log \| A_i(X_0)A_i(X_1)\dots A_i(X_{t-1})\| \] \red{exists} for $\mu$-almost every $x \in \R^n_+$ \red{and} \[\rr_i(\mu)=\int \rr_i(x)\mu(dx) \] \red{which we call  the long-term growth rate of species $i$ with respect to $\mu$}. When $\mu$ is ergodic, the subadditive ergodic theorem implies that $\rr_i(x)$ equals $\rr_i(\mu)$  for $\mu$-almost every $x \in \R^n_+$.

A \red{compact} set $M\subset S_A$ is said to be \emph{invariant} for $\Phi_A$ if $\Phi_A(M)=M$ \red{and is \emph{isolated} if there exists a compact neighborhood $N$, called \emph{an isolating neighborhood of $M$}, such that $M$ is the largest invariant set in $N$. For sufficiently small $\delta$-perturbations $\Phi_{A^\delta}$ of $\Phi_A$, an isolating neighborhood $N$ of $M$ for $\Phi_A$ is also an isolating neighborhood for a compact invariant set $M(\delta)$ of $\Phi_{A^\delta}$. $M(\delta)$ is called the \emph{continuation} of $M$. }

\red{Recall the $\omega$-limit set of a point $x\in S_A$ is defined by
\[
\omega_A(x)=\{y \in S_A: \mbox{ there exists } t_k\uparrow \infty \mbox{ such that } \lim_{k\to\infty} \Phi_A^{t_k}(x) =y \}
\]
For a compact invariant set $M$, let $W^s(M)=\{x\in \mathbb{R}^n_+\: \mid \: \omega_A(x)\subset M\}$ be the \emph{stable set} of $M$}.

\begin{definition} A compact set $M\subset S_0$ is \emph{unsaturated} for model \eqref{DYN1} if $M$ is isolated in $\mathbb{R}^n_+$ and $W^s(M)\subset S_0$. If, in addition, the continuation of $M$ corresponding to \red{sufficiently} small \red{$\delta$-}perturbations of model \eqref{DYN1} \red{are} unsaturated, we call $M$ \emph{robustly unsaturated}. \end{definition}

With these definitions, we can state our main Theorem.

\begin{theoreme}\label{thm:main}
Let $M$ be a compact isolated set for the dynamics $\Phi_A$ restricted to $S_0$. If one of the following equivalent conditions holds
\begin{enumerate}
\item[(i)] $\rr_*(\mu) := \max_{1\le i\le m} \rr_i(\mu)>0$
for every $\Phi_A$-invariant probability measure with $\mu (M)=1$, or
\item[(ii)] there exist positive constants $p_1,\dots,p_m$ such that
\[
\sum_i p_i \rr_i(\mu)>0
\]
for every ergodic probability measure with $\mu(M)=1$, or
\item[(iii)] there exist positive constants $p_1,\dots, p_m$ such that
\[
\sum_i p_i \rr_i(x)> 0
\]
for all $x\in M$,
\end{enumerate}
then $M$ is robustly unsaturated for model	 \eqref{DYN1}.
\end{theoreme}

\red{
\begin{remark}\label{remark}
For some applications (e.g., the disease model considered in section~\ref{application:disease}), it is useful to relax the primitivity assumption \textbf{H2}. For instance, if there exists an open neighborhood $U$ of $M$ such that for each $i$, $A_i(x)$ has a fixed sign pattern for all $x\in U$, then $A_i(x)$ can be decomposed into a finite number, say $m_i$, of irreducible components. For each of these irreducible components, one can define $\rr_i^j (x) =\limsup_{t\to\infty} \frac{1}{t} \log \| A_i^j(X_0)A_i^j(X_1)\dots A_i^j(X_{t-1})\|$ with $X_0=x$ where $A_i^j(x)$ is the submatrix of $A_i(x)$ corresponding to the $j$-th irreducible component of $A_i(x)$. If we define $\rr_i(x)=\max_{1\le j\le m_i} \rr_i^j (x)$ and similarly define $\rr_i^j(\mu)$ and $\rr_i(\mu)=\max_{1\le j\le m_i} \rr_i^j (\mu)$, then all of the assertions of Theorem~\ref{thm:main} still hold.
\end{remark}
}

We can prove a partial converse to Theorem~\ref{thm:main}. \red{For any ergodic probability measure $\mu$, define $\mbox{species}(\mu)$ to be the unique subset of   $\{1,2,\dots,n\}$ such that \[\mu(\{x: \|x_i\|>0 \mbox{ if and only if }i\in \mbox{species}(\mu)\})=1.\]}
\begin{proposition}\label{prop:nec}
Assume $\Phi_A$ is twice continuously differentiable. Let $M$ be a compact isolated set for the dynamics $\Phi_A$ restricted to $S_0$. If $\mu$ is an ergodic, probability measure with $\mu(M)=1$ and
\[
\rr_i(\mu)\le 0 \mbox{ for all }i\notin \red{\mbox{species}(\mu)},
\]
then $M$ is not robustly unsaturated for \eqref{DYN1}. Specifically, for any $\delta>0$, there exists a $\delta$-perturbation $A^\delta$ of \eqref{DYN1}  and  $x\in S_{A^\delta}\setminus S_0$ such that $\omega_{A^\delta}(x)\subset M$.
\end{proposition}

\begin{proof}
Choose $\delta>0$. Let $U\subset N_{\delta/2}(S_A)$ be an open neighborhood of $S_A$ such that $\overline{U}$ is forward invariant for \eqref{DYN1}. Let $V\subset W$ be open neighborhoods of $M$ such that $\overline{W}\subset U$. Let $\psi: \R^n_+\to [0,1]$ be a smooth function such that $\psi(x)=1$ for all $x\in V$ and $\psi(x)=0$ for all $x\in \R^n_+\setminus W$. Define
\[
A_i^\delta(x)=
\left\{
\begin{array}{ll}
A_i(x)\red{\exp(-\psi(x)\delta/2)}&\mbox{ if }i\notin \red{\mbox{species}(\mu)}\\
A_i(x) & \mbox{ if } i\in \red{\mbox{species}(\mu)}.
\end{array}
\right.
\]
By construction, $A^\delta$ is a $\delta$-perturbation of the model (\ref{DYN1}) associated with $A$, $M$ is a compact invariant set and $\mu$ is an ergodic measure for the dynamics $X_{t+1}=X_t A^\delta(X_t)$, and
\[
\rr_i(\mu) \le -\frac{\delta}{2} \mbox{ if } i \notin \red{\mbox{species}(\mu)}
\]
for the dynamics of $X_{t+1}=X_tA^\delta(X_t)$. By Ruelle and Shub's stable manifold theorem for maps~\citep{ruelle-shub-80}, there are points $x\in M$ such that the stable manifold of $x$ for the $\delta$-perturbation dynamics intersects the interior of $\R^n_+$.
\end{proof}

\subsection{Morse decompositions and robust permanence}

Here we state a sufficient condition for robust permanence using a characterization of permanence due to \cite{hofbauer-so-89} that involves Morse decompositions of boundary dynamics. \red{As $\Phi_A$ is not invertible, backward orbits of a point $x$ need not be unique. Consequently, a sequence $\{x_s\}_{s=0}^{-\infty}$ is a \emph{backward orbit} through $x$ if $x_0=x$ and $x_{-s+1}=\Phi_A(x_{-s})$ for all $s\ge 1$. The $\alpha$-limit set of a backward orbit $\{x_s\}_{s=0}^{-\infty}$ is $\alpha(\{x_s\})=\{y\,:\, \lim_{s_k\to -\infty} x_{s_k}=y$ for some $s_k\to - \infty\}$. } Following
\cite{hofbauer-so-89}, we define a collection of sets $\{M_1,\dots,M_k\}$ to be a \emph{Morse decomposition} for a compact invariant set $M$ if

\begin{itemize}
\item $M_1,\dots,M_k$ are pairwise disjoint, compact isolated sets for $\Phi_A$ restricted to $M$.
\item \red{For each $x\in M\setminus \cup_{i=1}^k M_i$, there is an $i$ such that $\omega(x)\subset M_i$ and for any backward orbit $\{x_s\}$ through $x$ there is a $j>i$ such that $\alpha(\{x_s\})\subset M_j$.}
\end{itemize}

 \cite{hofbauer-so-89} proved the following characterization of permanence.

\begin{theoreme}[Hofbauer \& So 1989]\label{garay-hof-so}
If $\{M_1,\dots,M_k\}$ is a Morse decomposition for $S_0 \cap \red{G_A}$, then model (\ref{DYN1}) is permanent if and only if each of the components $M_i$ are unsaturated.
\end{theoreme}

Then Theorem \ref{thm:main} and \ref{garay-hof-so} imply the following result:

\begin{theoreme}\label{thm:robust}
If $\{M_1,\dots,M_k\}$ is a Morse decomposition for $S_0\red{\cap G_A}$ and condition (i) of Theorem \ref{thm:main} holds for each of the components of the Morse decomposition, then model \eqref{DYN1} is robustly permanent.
\end{theoreme}

The following lemma is used in the proof of Theorem \ref{thm:robust}
\begin{lemma}\label{lemma:robust}
\blue{Let $M\subset S_0$ be an isolated set and $\{M_1,\dots,M_k\}$ be a Morse decomposition of $M$ for the map $\Phi_A$ restricted to $S_0$. Then, for sufficiently small $\delta>0$, there is an non empty subset $\{i_1,\dots,i_l\}\subset \{1,\dots,k\}$ such that the set of continuation $\{M_{i_1}(\delta),\dots,M_{i_l}(\delta)\}$ is a Morse decomposition of the continuation $M(\delta)$ for the map $\Phi_{A^\delta}$ restricted to $S_0$.}
\end{lemma}

\begin{proof}
Let $M\subset S_0$ be an isolated subset and $\{M_1,\dots,M_k\}$ be a Morse decomposition of $M$ for the map $\Phi_A$ restricted to $S_0$. Then, from Theorem~3.10 in \cite{patrao-07}, there exists a strictly increasing sequence of attractors in $M$

\begin{equation}\label{eq:Robust_1}
  \emptyset=A_0 \subset A_1\subset ... \subset A_k=M,
\end{equation}

\noindent
such that

\begin{equation*}
  M_i=A_i\cap A_{i-1}^*,\: i=1,...,k,
\end{equation*}

\noindent
where $A_i^*=\{x\in M\mid \omega(x)\cap A_i=\emptyset\}$ is the repeller corresponding to $A_i$. Both $A_i$ and $A_{i-1}^*$ are compact, isolated invariant sets, for each $i=1,...,k$.

For sufficiently small $\delta>0$, $\left(A_i({\delta}),A_i^*({\delta})\right)$ is an attractor-repeller pair for the map $\Phi_{A^\delta}$ in $S_0$. \blue{This is a consequence of Theorem 5 in \cite{mischaikow-99} which remains valid for maps}. Define

\begin{equation*}
  M_i^{\delta}=A_i({\delta})\cap {A_{i-1}^*}({\delta}),\: i=1,...,k,
\end{equation*}

\noindent
and let $N_i$ and $N_{i-1}^*$ be isolating neighborhoods of $A_i$ and $A_{i-1}^*$, respectively. In particular, $N_i \cap N_{i-1}^*$ is an isolated neighborhood of $M_i$. By definition of a continuation, for sufficiently small $\delta$, $A_i(\delta)$ and $A_{i-1}^*(\delta)$ are the largest compact, invariant sets in $N_i$ and $N_{i-1}^*$, respectively. Hence $M_i^{\delta}$ is the continuation $M_i(\delta)$ of $M_i$.
Equation (\ref{eq:Robust_1}) implies that
\begin{equation*}
  N_0 \subset N_1\subset ... \subset N_k,
\end{equation*}
which implies that
\begin{equation}\label{eq:robust2}
  \emptyset=A_0(\delta) \subset A_1(\delta)\subset ... \subset A_{k}(\delta)=M(\delta).
\end{equation}
Note that the first inclusion in equation (\ref{eq:robust2}) is strict, since $M_1=A_1$ is an attractor and then it has a nonempty continuation. \blue{Hence, there exists a non-empty subset $\{i_1,...,i_l\}\subset \{1,...,k\}$ such that the inclusions in (\ref{eq:robust2}) restricted to the indices in this subset are strict.} Thus, again from Theorem~3.10 in \cite{patrao-07}, $\{M_{i_1}(\delta),\dots,M_{i_l}(\delta)\}$ is a Morse decomposition of $M(\delta)$ for the map $\Phi_{A^\delta}$.

\end{proof}
\begin{proof}[Proof of Theorem \ref{thm:robust}]
Theorems \ref{thm:main} and \ref{garay-hof-so} and Lemma \ref{lemma:robust} imply that, for sufficiently small $\delta>0$, the model \eqref{DYN1}, with $A$ replaced by $A^{\delta}$, is permanent. \blue{Let $R=G_A\cap S_0$ and $G_A^0=G_A\backslash S_0$. Since the model (\ref{DYN1}) is permanent, $G_A^0$ is the interior global attractor of $\Phi_A$. Then we have that $(G^0_A,R)$ is an attractor-repeller pair for $G_A$. Let $N\subset \mathbb{R}^n_+$ with $N\cap S_0=\emptyset$ be an isolating neighborhood of $G^0_{A}$. Hence, for sufficiently small $\delta$, $(G^0_{A}(\delta),R(\delta))$ is an attractor-repeller pair for $G_A(\delta)$, and $G^0_{A}(\delta) \subset N$. Since the model (\ref{DYN1}), corresponding to $\delta$, is permanent, $R(\delta)=G_A(\delta)\cap S_0$. Hence there is a common region of repulsion around the boundary which concludes the proof.}
\end{proof}

For two species models, Theorem~\ref{thm:robust} and Proposition~\ref{prop:nec} provide a precise characterization of robust permanence. The twice continuously differentiable assumption is only required to show the necessity of the conditions of the proposition for robust permanence.

\begin{corollary}\label{for:2species} Assume $ m=2$ and $x\mapsto x A(x)$ is twice continuously differentiable. Then model \eqref{DYN1} is robustly permanent if and only if
\begin{itemize}
\item $\max_i \rr_i (0)>0$,
\item $\rr_2(\mu)>0$ for any ergodic measure $\mu$ with $\red{\mbox{species}(\mu)}=\{1\}$, and
\item $\rr_1(\mu)>0$ for any ergodic measure $\mu$ with $\red{\mbox{species}(\mu)}=\{2\}$.
\end{itemize}
\end{corollary}

\begin{proof}
Suppose that the three conditions hold. Since $\max_i \rr_i (0)>0$, we can choose positive $p_1,p_2$ such that $\sum_i p_i \rr_i(0)>0$. Let $\mu$ be an ergodic measure $\mu$ supported in $\red{G_A}\cap S_0$. We will show that $\sum_i p_i \rr_i(\mu)>0$. If $\mu=\delta_0$, then we are done \red{as $r_i(\mu)=r_i(0)$}. If $\mu\neq\delta_0$, then $\red{\mbox{species}(\mu)}=\{j\}$ for some $j\in \{1,2\}$. Since $\rr_j(\mu)=0$, we have $\sum_i p_i \rr_i(\mu)=p_\ell\rr_\ell(\mu)$ where $\ell\neq j$. By the second and third conditions, $\rr_\ell(\mu)>0$. Hence, $\sum_i p_i \rr_i(\mu)>0$ for all ergodic $\mu$ supported in $S_0\red{\cap G_A}$. Applying Theorem~\ref{thm:robust} with Morse decomposition $M_1= S_0\red{\cap G_A}$ completes the proof of this implication.

Now suppose one of the conditions doesn't hold. Then Proposition~\ref{prop:nec} with $M=S_0\red{\cap G_A}$ implies model~\eqref{DYN1} is not robustly permanent.
\end{proof}

Theorem~\ref{thm:robust} and Proposition~\ref{prop:nec} also characterize models \eqref{DYN1} for which the model and its restriction to any subset of species is robustly permanent. That is, for any non-empty set $I\subseteq \{1,...,m\}$, the system \eqref{DYN1} restricted to $\prod_{i\in I}\mathbb{R}^{n_i}_+$ is robustly permanent. This characterization is the discrete-time extension of \citep[Theorem 3.3]{jmaa-02} to structured population models. In particular, we have the following

\begin{corollary}
Model \eqref{DYN1} and all of its submodels are robustly permanent if, for all ergodic probability measures $\mu$ with support in $S_0$,
\begin{equation}\label{eq:Subsyst1}
\rr_i(\mu)>0 \ \text{ for all } i \in \{1,\dots,m\} \backslash \red{\mbox{species}(\mu)}.
\end{equation}

Conversely, if $\Phi_A$ is twice continuously differentiable and \eqref{DYN1} and all of its subsystems are robustly permanent then, for all ergodic probability measures $\mu$ with support in $S_0$, \eqref{eq:Subsyst1} holds.

\end{corollary}

\begin{proof}
Let $\emptyset\neq I\subseteq \{1,...,m\}$. Let $S_A^I=\{x\in S_A \mid x^i=0,\: \forall \: i\not \in I\}$ and $S_0^I=\{x\in S_A^I \mid \prod_{i\in I}||x^i||=0\}$. Consider the restriction of \eqref{DYN1} to $S^I_A$:

\begin{equation}\label{DYN_Restrict}
X_{t+1}^i=X_t^i A_i(X_t), \: i\in I.
\end{equation}

\noindent
Let $\{M_1,...,M_k\}$ be a Morse decomposition of $S_0^I\red{\cap G_A}$. Fix $j\in \{1,\dots,k\}$, and let $\mu$ be an ergodic probability measure with $\mu(M_j)=1$. Note that $\red{\mbox{species}(\mu)}\neq I$. Then by Proposition \ref{prop_lambdanul}, for every $i\in \red{\mbox{species}(\mu)}$ we have that $r_i(\mu)=0$, while for every $i\in I\setminus \red{\mbox{species}(\mu)}$ we have $r_i(\mu)>0$ (by assumption \eqref{eq:Subsyst1}). Thus, condition (ii) in Theorem~\ref{thm:main} holds with $p_i=1$ for all $i\in I$. Hence $M_j$ is robustly unsaturated with respect to \eqref{DYN_Restrict}. Theorem~\ref{thm:robust} implies \eqref{DYN_Restrict} is robustly permanent.

For the converse, suppose that there exists $\mu$ an ergodic probability measure with support in $S_0$ such that $r_l(\mu)\leq 0$ for some $l \in \{1,\dots,m\} \backslash \red{\mbox{species}(\mu)}$. Let $I=\red{\mbox{species}(\mu)}\cup \{l\}$.  Proposition~\ref{prop:nec} implies \eqref{DYN_Restrict} is not robustly permanent, a contradiction.

\end{proof}
\section{Proof}
Recall some useful definitions and notations. For $d \in \N$, let $\interior \R_+^d = \{x\in \R^d_+ : \prod_i x_i >0\}$ be the interior of $\R^d_+$ and $\mathbf{M}_{d}(\R)$ be the set of all $d\times d$ matrices over $\R$. Let $M$ be a metric space, and let $\PP(M)$ be the space of Borel probability measures on $M$ endowed with the weak$^*$ toplogy. The support of a measure $\nu$ in $\PP (M)$, denoted by $\support(\nu)$, is the smallest closed \red{set $B\subset \R^n_+$ such that $\mu(B)=1$}. If $M'$ is also a metric space and $f \colon M \to M'$ is Borel measurable, then the induced linear map $f^{*} \colon \mathcal{P}(M) \to \mathcal{P}(M')$ associates with $\nu \in \mathcal{P}(M)$ the measure $f^{\ast}(\nu) \in \mathcal{P}(M')$ defined by
\[
f^\ast(\nu) (B)=\nu(f^{-1}(B))
\]
for all  Borel sets $B$ in $M'$.
If $\theta \colon M \rightarrow M$  is a continuous map, a measure $\nu \in \mathcal{P}(M)$ is called \emph{$\theta$-invariant} if $\nu (\theta^{-1}(B)) = \nu(B)$ for all Borel sets $B \in M$. A set $B\subset M$ is \emph{positively invariant} if $\theta(B) \subset B$. For every positively invariant compact set $B$, let $\PPinv(B)$ be the set of all $\theta$-invariant measures supported on $B$.

Given $x \in \R^n_+$, the \emph{empirical occupation measure} at time $t\in \R_+$ of $\{\Phi^s_A\}_{s\ge 0}$ is
\[
\Lambda_t(x) := \frac{1}{t} \sum_{s=0}^{t-1}\delta_{\Phi^s_A(x)}.
\]
where $\delta_y$ denotes a Dirac measure at $y$, i.e. $\delta_{y}(A) =1$ if $y\in A$ and $0$ otherwise for any Borel set $A \subset \R_+^n$. These empirical measures describe the distribution of the observed population dynamics up to time $t$. In particular, for any Borel set $B\subset \R^n_+$,
\[
\Lambda_t(x)(B)= \frac{\#\{ 0\le s \le t-1 | \Phi^s_A(x) \in B \}}{t}
\]
is the fraction of time that the populations spent in the set $B$.

The following propositions and lemma are crucial for the proof of Theorem~\ref{thm:main}.

\begin{proposition}\label{proposition-convergence-empiricalmeasures}
For all $x \in \interior  \R^n_+ $, every weak$^*$ limit point $\mu$ of the family of probability measures $\{\Lambda_t(x)\}_{t\in \N}$ belongs to $\Inv(\Phi_{A})(\red{G_A})$ and satisfies $\rr_i(\mu)\le 0$ for all $i$.
\end{proposition}

\begin{lemma}\label{lemma-inequality}
\red{Let} $x \in S_0$ \red{and $\mu$ be a  weak$^*$ limit point of the family of probability measures $\{\Lambda_t(x)\}_{t\in \N}$. Then}
\[
\rr_i(x) \ge \rr_i(\mu) \mbox{ \red{ for $i=1,\dots,m$.}}
\]
\end{lemma}

We call an invariant measure $\mu$ for model \eqref{DYN1} \emph{saturated} if $\rr_i(\mu)\le0$ for all $i$.

\begin{proposition}\label{proposition:robust-persistence}
Let $(\delta_n)_{n\ge 1}$ be a non-negative sequence that converges to zero, and $(A^n)_{n\ge0}$ be a sequence of $\delta_n$-perturbations of model (\ref{DYN1}). Let $\{\mu_n\}_{n\ge 1}$ be saturated $\Phi_{A^n}$-invariant measures, then the weak* limit points of $\{\mu_n\}_{n\ge1}$ is a non-empty set consisting of saturated $\Phi_A$-invariant measures.
\end{proposition}
The proofs of Proposition~\ref{proposition-convergence-empiricalmeasures}, Lemma~\ref{lemma-inequality}, and Proposition~\ref{proposition:robust-persistence} are postponed to the end of the section. We now prove Theorem~\ref{thm:main}.

\begin{proof}
First we show that $(i)\Leftrightarrow (ii)\Leftrightarrow (iii)$. For $(i)\Leftrightarrow (ii)$ see \cite{hofbauer-schreiber-10}. $(ii)$ is obtained by integrating the inequality in $(iii)$. Finally, we prove that $(ii)\Rightarrow (iii)$. Thus, let $x\in S_0$ and $\mu=\lim_{k\to \infty}\Lambda_{t_k}(x)$ be a weak$^*$ limit point of the sequence $\{\Lambda_{t}(x)\}_{t\ge1} \in \Inv(\Phi_A)(S_0)$. By Lemma~\ref{lemma-inequality}, $\rr_i(x) \ge \rr_i(\mu)$ for all $i=1,\dots,m$. Writing $\mu$ as a convex combination of ergodic probability measures, condition $(ii)$ implies condition $(iii)$.

Now we show $(i)$ implies that $\red{M}$ is unsaturated, arguing by contradiction. Suppose that $(i)$ holds and that $\red{M}$ is saturated. Theorem~2.1 in \cite{hofbauer-so-89} implies that either $W^s(\red{M})$, the stable \red{set} of $\red{M}$, contains points in $\red{G_A}\setminus S_0$, or $\red{M}$ is not isolated in $\mathbb{R}^n_+$. We \red{will show} that in either case there exists a saturated $\Phi_A$ invariant measure with support in $\red{M}$ which contradicts $(i)$.

Consider the first case \red{where} there exists an $x\in W^s(\red{M})\setminus S_0$. Let $\mu$ be a weak$^*$ limit point of $\{\Lambda_t(x)\}_{t\ge1}$. Proposition~\ref{proposition-convergence-empiricalmeasures} implies that $\mu$ is a saturated $\Phi_A$ invariant measure. On the other hand, since  $x\in W^s(\red{M})$, $\support(\mu) \subset \red{M}$.

Now consider the second case \red{where $M$} in not isolated in $\mathbb{R}^n_+$. Then there exists a sequence $\{\omega_n\}_{n\in \mathbb{N}}$ of omega limit sets of points in $\red{G_A}\setminus S_0$ that accumulate on $\red{M}$. Let $\{\mu_n\}_{n\in \mathbb{N}}$ be a corresponding sequence of ergodic probability measures supported by these omega limit sets. Then, Lemma~\ref{lemma:inv-measure}, Propositions~\ref{prop_lambdanul} and Proposition~\ref{prop:equality-lyap} imply that $r_i(\mu_n)= 0$ for all $i=1,...,m$ and all $n\in \mathbb{N}$. Hence the $\mu_n$'s are saturated measures. Proposition~\ref{proposition:robust-persistence}, applied with $\delta_n=0$ for all $n$, \red{implies} that $\{\mu_n\}_{n\in \mathbb{N}}$ has a weak$^*$ limit point $\tilde{\mu}$ that is a saturated $\Phi_A$ invariant measure with support \red{in $M$}. This concludes the proof of the claim.

Finally, we show that $\red{M}$ is robustly unsaturated. If $\red{M}$ is not robustly unsaturated, there exists a sequence $\{\delta_n\}_{n\in \mathbb{N}}\subset \mathbb{R}_+$ and a corresponding sequence $\{\mu_n\}_{n\in \mathbb{N}}$ of saturated measures for the $\delta_n$-perturbations of model \eqref{DYN1} with support in the continuation of $\red{M}$. So again, from Proposition~\ref{proposition:robust-persistence}$, \{\mu_n\}_{n\in \mathbb{N}}$ has a weak$^*$ limit point $\tilde{\mu}$ that is a saturated $\Phi_A$ invariant measure. Moreover, since the continuation of $\red{M}$ converges to $\red{M}$ as $\delta_n \rightarrow 0$, $\tilde{\mu}$ is supported by $\red{M}$, which contradicts $(i)$.

\end{proof}

The rest of this section is dedicated to the proofs of Propositions~\ref{proposition-convergence-empiricalmeasures} and \ref{proposition:robust-persistence}, and Lemma~\ref{lemma-inequality}. \red{To this end,} define the space
\[
\mathcal{D}:=\{B: \R^n_+ \rightarrow \mathbf{M}_{n}(\R) : B\text{ induces  a } \delta\text{-perturbation of model} (\ref{DYN1})\}
\]
endowed with the pseudo-metric induced by the norm sup on the compact $N_1(S_A)$. Define the map $\delta : \mathcal{D} \rightarrow \R_+$ by $\delta(d):=\inf\{\delta \in \R_+ : d \text{ is a } \delta \text{-perturbation}\}$ for $d\in \mathcal{D}$.

In order to prove the robustness result, we need to link the long-term behavior of model (\ref{DYN1}) with the long-term behavior of $\delta$-perturbations with small $\delta$. To that end, we regroup those dynamics under one dynamics over a larger space. Let $(A^{n})_{n\ge 0} \subset \mathcal{D}$ be a sequence of $\delta(A^n)$-perturbations such that $\delta(A^n) \rightarrow 0$ as $n\rightarrow \infty$. Define the set $C:= \overline{\{A^n : \ n \ge 0\}} \subset \mathcal{D}$. The set $C$ is compact. Indeed, since $\delta(A^n) \downarrow 0$, properties (i) and (ii) of a $\delta$-perturbation imply that the sequence $\{A^n\}_n$ converge uniformly to $A$ in $N_1(S_A)$. Define
\begin{equation}\label{def:phi}
\begin{array}{ccl}
 \Phi : &\R^n_+ \times \mathcal{D}   &\rightarrow  \R^n_+ \times \mathcal{D}  \\
  & x, c  & \mapsto  xc(x),c
\end{array}
\end{equation}
and the projection map $p: \R^n_+ \times \mathcal{D}  \rightarrow  \R^n_+$ as $p(x,c)=x$. Write $\Phi_c$ for $p\circ \Phi(\cdot,c)$. Note that $\Phi_A$ is consistent with (\ref{def:phiA}). Let $\Phi^t_c$ denote the composition of $\Phi_c$ with itself $t$ times, for $t \in \N$ and $c\in \mathcal{D}$.  Hence model (\ref{DYN1}) can be rewritten as
\[
X_{t+1}(x) = \Phi_A^{t+1}(x).
\]

Assumption \textbf{H3} can be rewritten in term of attractor of the dynamics induced by $\Phi_A$.

\begin{definition}\label{def:attractor}
A compact set $K\subset \R^n_+$ is a \emph{global attractor for $\Phi_A$} if there exists a neighborhood $V$ of $K$ such that
\begin{enumerate}
\item[(i)] for all $x \in  \R^n_+ $, there exist $T \in \N$ such that $\Phi_A^t(x) \in V$ for all $t \ge T$;
\item[(ii)] $\Phi_A(V) \subset V$ and $K=\bigcap_{t\in \N}\Phi_A^t(V)$.
\end{enumerate}

\end{definition}

Assumption \textbf{H3} takes on the form
\begin{enumerate}
\item[\textbf{H3':}] There exists a global attractor $S_A \subset \R^n_+$ for $\Phi_A$.
\end{enumerate}


 \subsection{Trajectory space\label{partB}}

The key element of the proof of Propositions~\ref{proposition-convergence-empiricalmeasures} and \ref{proposition:robust-persistence} is Proposition \ref{prop:ruelle} due to \cite{ruelle-79b} in which it is crucial that the map $\Phi$ is an homeomorphism. For now, it is not the case. Indeed, the maps $\Phi_c$ are, a priori, not invertible. To avoid this constraint we extend the dynamics induced by $\Phi$ to an invertible dynamics on the larger set of possible trajectories.

Assumption \textbf{H3'} and the fact that $\delta(A^n) \downarrow 0$ imply that, for each $c\in C$, there exist a non empty closed set $V_c$ such that (i) $V_c \subset N_1(S_A)$ and (ii) $\Phi_c(V_c) \subset  V_c$. Without loss of generality, we assume that $V_c$ is the larger open set such that (i) and (ii) \red{hold}.

Fix $c\in C$. Property (ii) of $V_c$ implies that, for every point $x \in V_c$, there exists a sequence $\{x_t\}_{t\in \N} \subset V_c$ such that $x_0=x$, and $x_{t+1}=\Phi_c(x_t)$ for all $t \ge 0$. Such a sequence is called a \emph{$\Phi_c$-positive trajectory}. In order to create a \emph{past} for all those $\Phi_c$-positive trajectories, let us pick a point $x^*\in \interior \R^n_+ \backslash N_1(S_A)$, and consider the sequence space $\mathcal{T}:=(N_1(S_A) \cup \{x^*\})^{\Z}$ endowed with the product topology, and the homeomorphism $\vphi : \mathcal{T} \rightarrow \mathcal{T}$ called \emph{shift operator}, and defined by $\vphi(\{x_t\}_{t\in \Z}) = \{x_{t+1}\}_{t\in \Z}$. Since $N_1(S_A) \cup \{x^*\}$ is compact, the space $\mathcal{T}$ is compact as well.

Every $\Phi_c$-positive trajectory can be seen as an element of $\mathcal{T}$ by creating a fixed past (i.e. $x_t=x^*$ for all $t<0$). Define $G_c\subset \mathcal{T}$ as the set of such $\Phi_c$-positive trajectories, and  $E_c =\overline{ \bigcup_{t\in \Z}\vphi^t(G_c)}\subset \mathcal{T}$. In words, $E_c$ is the adherence in $\mathcal{T}$ of the set of all shifted (by $\vphi^t$ for some $t\in \Z$) $\Phi_c$-positive trajectories. Since $E_c$ is a closed subset of the compact $\mathcal{T}$, it is compact as well. Define
 \[
 \Gamma := \bigcup_{c\in C}\left(E_c \times \{c\}\right),
 \]
subset of the product space $\mathcal{T}\times C$. From now on, when we write $(\bbx,c) \in \Gamma$, we mean $\bbx=\{x_t\}_{t\in \Z}\in E_c$ and $c\in C$.

\begin{lemma}\label{lemma:gammaclosed}
 $\Gamma$ is a compact subset of $\mathcal{T}\times C$.
\end{lemma}
\begin{proof}
Since $\mathcal{T}\times C$ is compact, we need only to show that $\red{\Gamma}$ is closed. Let $(\bbx,c) \in \mathcal{T} \times C$ and $\{(\bbx_n,c_n)\}_{n\ge 0} \subset \Gamma$ be a sequence converging to $(\bbx,c)$. By definition of $C$, we need only to consider the case $c=A$ and show that $\bbx \in E_A$. Define the closed set $W:=\{x \in N_1(S_A) : \exists (n_k)_k \uparrow \infty, (x_k)_k \in \mathcal{T} \text{ s.t. }  x_k \in V_{c_{n_k}} \forall k \text{ and } x= \lim_{k \to \infty}x_{k} \}$. We claim that
\[
\Phi_A(W) \subset W.
\]
Let $x\in W$ and $(x_k)_k \in \mathcal{T}$ such that $x_k \in V_{c_{n_k}}$ for all $ k$ and $ x= \lim_{k \to \infty}x_{k}$. Fix $\eps>0$. For $k$ large enough $\vert \Phi_A(x)-\Phi_{c_{n_k}}(x)\vert \le \eps$ and by equicontinuity of $C$, $\vert \Phi_{c_{n_k}}(x)-\Phi_{c_{n_k}}(x_k)\vert \le \eps$. Then
\begin{eqnarray*}
\vert \Phi_A(x)-\Phi_{c_{n_k}}(x_k)\vert &\le& \vert \Phi_c(x)-\Phi_{c_{n_k}}(x)\vert + \vert \Phi_{c_{n_k}}(x)-\Phi_{c_{n_k}}(x_k)\vert \\
&\le& 2\eps.
\end{eqnarray*}
Hence $\Phi_{c_{n_k}}(x_k) \rightarrow \Phi_A(x)$ as $k \rightarrow \infty$. Since $\Phi_{c_{n_k}}(x_k) \in V_{c_{n_k}} \subset N_1(S_A)$ for all $k$, $\Phi_c(x)\in W$ which proves the claim.

The maximality of $V_A$ and the claim imply that $W \subset V_A$. If $\bbx_n =(x_k^n)_{k\ge0}$, and $\bbx = (x_k)_{k\ge 0}$, then for all $k\ge 0$, $x_k^n \rightarrow x_k$ as $n \rightarrow \infty$. Then for each $k\ge0$, either $x_k = x^*$, or $x_k \in W \subset V_A$. If $x_k \in V_A$, equicontinuity of $C$ implies that $\Phi_A(x_k)=x_{k+1}$. Hence $\bbx \in E_A$. This concludes the proof.
\end{proof}
Define the homeomorphism
\begin{equation*}
\begin{array}{ccl}
 \Theta: &\mathcal{T} \times C   &\rightarrow \mathcal{T} \times C    \\
  & \bbx, c  & \mapsto  \vphi(\bbx),c
\end{array}
\end{equation*}

By definition of the sets $E_c$, $E_c \times \{c\}$ are invariant under $\Theta$. In particular, $\Gamma$ is invariant under $\Theta$, which implies that the restriction $\restr{\Theta}{\Gamma}$ of $\Theta$ \red{to} $\Gamma$ is well-defined. To simplify the presentation we still denote this restriction by $\Theta$. The projection map $\pi_0 : \Gamma \rightarrow N_1(S_A) \cup \{x^*\} \times C$ is defined by $\pi_0(\bbx,c)=(x_0,c)$ for all $(\bbx,c) \in \Gamma$. The map $\pi_0$ is continuous and $\pi_0(\Gamma)= \bigcup_{c\in C}\left(V_c \cup \{x^*\} \times \{c\}\right)$.

Next, for all $c\in C$, we define the compact set of all \emph{$\Phi_c$-total trajectories} as
\[
\Gamma_+^c:= \pi_0^{-1}(V_c \times \{c\}),
\]
and the compact set of \emph{$\Phi_c$-total trajectories on the extinction set $\bS_0$} as
\[
\Gamma_0^c:= \pi_0^{-1}(\bS_0 \times \{c\}).
\]
The respective union over $C$ of those sets of trajectories are $\Gamma_+ := \bigcup_{c\in C}\Gamma_+^c$ and $\Gamma_0 := \bigcup_{c\in C}\Gamma_0^c$.

For all $c\in C$, the dynamic induced by $\Phi_c$ on $V_c \times \{c\}$ is linked to the dynamic induced by $\Theta$ on $\Gamma_+^c$ by the following semi conjugacy
 \begin{equation}\label{Eq_conj}
p \circ \pi_0 \circ \Theta = \Phi_c \circ p \circ \pi_0.
\end{equation}

As a consequence of the semi-conjugacy (\ref{Eq_conj}), we show that the statistical behavior of $\restr{\Theta}{\Gamma_+^c}$ and $\Phi_c$ are linked, and the set of $\Phi_c$-invariant measures is the projection through $p\circ \pi_0$ of the set of $\restr{\Theta}{\Gamma_+^c}$-invariant measures.

Given a trajectory $\gamma \in \Gamma_+$, the \emph{empirical occupation measure} at time $t\in \R_+$ of $\{\Theta^s(\gamma)\}_{s\ge 0}$ is
\[
\tilde{\Lambda}_t(\gamma) := \frac{1}{t} \sum_{s=0}^{t-1}\delta_{\Theta^s(\gamma)},
\]
and given $(x,c) \in V_c \times C$, the \emph{empirical occupation measure} at time $t\in \R_+$ of $\{\Phi^s_c(x)\}_{s\ge 0}$ is
\[
\Lambda_t(x,c) := \frac{1}{t} \sum_{s=0}^{t-1}\delta_{\Phi^s_c(x)}.
\]

\begin{lemma}\label{lemma_mes_emp}
Let $(\bbx,c) \in \Gamma_+$. Then for all $t\ge 0$ we have
\[
(p \circ \pi_0)^*(\tilde{\Lambda}_t(\bbx,c))=\Lambda_t(x_0,c).
\]
\end{lemma}

\begin{proof}
Let $(\bbx,c) \in \Gamma_+$ and $B \subset  V_c$ be a Borel set. Then we have
\begin{eqnarray*}
(p\circ \pi_0)^*(\tilde{\Lambda}_t(\bbx,c))(B) &=& \tilde{\Lambda}_t(\bbx,c)((p \circ \pi_0)^{-1}(B))\\
&=& \frac{1}{t} \sum_{s=0}^{t-1}\delta_{\Theta^s(\bbx,c)}((p \circ \pi_0)^{-1}(B))\\
&=& \frac{1}{t} \sum_{s=0}^{t-1}\delta_{\Phi^s_c\circ p \circ \pi_0(\bbx,c))}(B)\\
&=&  \frac{1}{t} \sum_{s=0}^{t-1}\delta_{\Phi_c^s(x_0)}(B)\\
&=& \Lambda_t(x_0,c)(B).
\end{eqnarray*}
The third equality is a consequence of the semi conjugacy (\ref{Eq_conj}).
\end{proof}

Since $\Gamma_+$ and $\Gamma_0$ are positively invariant and compact sets, it follows from classical results in dynamical systems theory (see e.g. \cite{katok-hasselblatt-95}):
\begin{lemma}\label{prop:mesinvtheta}
$\Inv(\Theta)(\Gamma_+)$ and $\Inv(\Theta)(\Gamma_0)$ are compact and convex subsets of $\mathcal{P}(\Gamma)$.
\end{lemma}

Since $\Gamma_+^c$ is positively $\Theta$-invariant and compact for all $c\in C$, Theorem 6.9 in \cite{walters-82} implies
 \begin{lemma}\label{prop_walters}
 For all $c\in C$ and all $\gamma \in \Gamma_+^c$, the set of all weak$^*$ limit point of the family of probability measures $\{\tilde{\Lambda}_t(\gamma)\}_{t\in \N}$ is a non-empty compact subset of $\Inv(\Theta)(\Gamma_+^c)$.
\end{lemma}

\begin{lemma}\label{lemma:inv-measure}
For all $c\in C$, $\Inv(\Phi_c)(V_c) = (p \circ \pi_0)^*(\Inv(\Theta)(\Gamma_+^c)$
\end{lemma}

\begin{proof}
Fix $c\in C$ . First we prove $\Inv(\Phi_c) \subset (p \circ \pi_0)^*(\Inv(\Theta))$. We show that the inclusion is satisfied for the set of $\Phi_c$~-~ergodic measure, and then the general case follows from the \emph{ergodic decomposition theorem}. Let $\mu \in \Inv(\Phi_c)$ be an ergodic measure. Since $\mu$ is ergodic, there exist $x\in V_c$ such that $\mu = \lim_{t \to \infty}\Lambda_t(x,c)$. There exists $(\bbx,c) \in \pi_0^{-1}(x,c) \subset \Gamma_+^c$. By compactness of $\mathcal{P}(\Gamma_+^c)$, let $\tilde{\mu}:= \lim_{k \to \infty}\tilde{\Lambda}_{t_k}(\bbx,c)$. Lemma \ref{prop_walters} implies that $\tilde{\mu} \in \Inv(\Theta)(\Gamma_+^c)$. Continuity of $p \circ \pi_0$ and Lemma~\ref{lemma_mes_emp} imply that $(p \circ \pi_0)^*(\tilde{\mu})=\mu$.

We now prove $\Inv(\Phi_c) \supset (p \circ \pi_0)^*(\Inv(\Theta))$. Let $\tilde{\mu} \in \Inv(\Theta)(\Gamma_+)$. Therefore the measure $(p \circ \pi_0)^*(\tilde{\mu})$ is supported by $V_c$. Let $B \subset V_c$ be a Borel set. We have
\begin{eqnarray*}
(p \circ \pi_0)^*(\tilde{\mu})(\Phi_c^{-1}(B)) &=&  \tilde{\mu}((p \circ \pi_0)^{-1}(\Phi_c^{-1}(B)))\\
&=& \tilde{\mu}((p \circ \pi_0)^{-1}(\Phi_c^{-1}(B) \cap V_c))\\
&=& \tilde{\mu}((\restr{\Phi_c}{V_c} \circ p \circ \pi_0)^{-1}(B))\\
&=& \tilde{\mu}((p\circ \pi_0\circ \restr{\Theta}{\Gamma_+^c})^{-1}(B))\\
&=& \tilde{\mu}((p\circ \pi_0)^{-1}(B))\\
&=& (p\circ\pi_0)^*(\tilde{\mu})(B).
\end{eqnarray*}
The second equality follows from the fact that the support of $\tilde{\mu}$ is inclued in $\Gamma_+$, and the fourth is a consequence of the conjugacy (\ref{Eq_conj}). This show that $(p \circ \pi_0)^*(\Inv(\Theta)(\Gamma_+^c))\subset \Inv(\Phi_c)(V_c)$, and then concludes the proof. \end{proof}

The map $\Theta$ on $\Gamma_+$ can be seen as the extension of the map $\Phi$ on $\bigcup_{c\in C}\left(V_c  \times \{c\}\right)$. Now, we define the long-term growth rates of the matrix products (\ref{matrixproducts}) over the extended dynamics $\Theta$.

Recall that by definition of a $\delta$-perturbation $c\in C$, there exist $c_i:\R^n\rightarrow \mathbf{M}_{n_i}(\R)$, for each $i\in\{1,\dots,m\}$ such that for each $x\in \R_+^n$, $c(x)=\diag(c_1(x),\dots,c_n(x))$.

For each $i\in \{1,\dots,m\}$, define the maps $\tilde{A}_i : \Gamma \rightarrow \mathbf{M}_{n_i}(\R)$ by
\[
\tilde{A}_i(\bbx,c)= c_i(x_0)
\]
We write
\begin{equation}\label{def:cocycle}
\tilde{A}^t_i(\bbx,c) := \tilde{A}_i(\bbx,c)\cdots \tilde{A}_i(\Theta^{t-1}(\bbx,c)).
\end{equation}
The conjugacy (\ref{Eq_conj}) implies that for all $c\in C$ and $x \in V_c$, we have
\begin{equation}\label{eq:cocycle}
\tilde{A}^t_i(\bbx,c)=c_i^t(x),
\end{equation}
for all $t\ge 0$ and all $\bbx \in \mathcal{T}$ such that $(\bbx,c) \in \pi_0^{-1}(x)$. Since there is no possible confusion, from now on we write $A_i$ instead of $\tilde{A}_i$.

Then the \emph{long-term growth rates} for the product (\ref{def:cocycle}) is
\[
\rr_i(\bbx,c) := \limsup_{t \to \infty}\frac{1}{t} \ln \Vert A^t_i(\bbx,c)\Vert,
\]
and, for a $\Theta$-invariant measure $\tilde{\mu}$, the \emph{long-term growth rates} is
\[
\rr_i(\tilde{\mu}) = \int_{\Gamma}\rr_i(\bbx,c)d\tilde{\mu}.
\]
\begin{proposition}\label{prop:equality-lyap}
For all species $i$ and all $c\in C$, we have
\begin{enumerate}
\item[(i)]  $\rr_i^c(x):= \limsup_{t\to\infty} \frac{1}{t} \log \| c_i(X_0)\dots c_i(X_{t-1})\| = \rr_i(\bbx,c)$, for all $x \in V_c$ and for all $\bbx \in \mathcal{T}$ such that $(\bbx,c) \in\pi_0^{-1}(x,c)$,
\item[(ii)] for all $\tilde{\mu} \in \Inv(\Theta)(\Gamma_+^c)$, we have \[
\rr_i(\tilde{\mu})=\rr_i((\pi_0\circ p)^*(\tilde{\mu})).
\]
\end{enumerate}

\end{proposition}
\begin{proof}
Assertion (i) is a consequence of equality (\ref{eq:cocycle}), and assertion (ii) is a consequence of assertion (i) and Lemma \ref{lemma:inv-measure}.
\end{proof}

\subsection{Properties of long-term growth rates\label{partC}}

In this section, we first state an extension of Proposition 3.2 of \cite{ruelle-79b} that has been proved in \cite{roth-schreiber-14}. We use this extension to deduce some properties on the long-term growth rates which are crucial for the proof of Propositions~\ref{proposition-convergence-empiricalmeasures} and \ref{proposition:robust-persistence}.

\begin{proposition}[Proposition 8.13 in \cite{roth-schreiber-14}]\label{prop:ruelle}
Let $\Xi$ be a compact space, $\Psi :  \Xi \rightarrow \Xi$ be an homeomorphism. Consider a continuous map $T :  \Xi \rightarrow \mathbf{M}_d(\R)$ and its transpose $T^*$ defined by $T^*(\xi)=T(\xi)^*$. Write
\[
T^t(\xi)= T(\xi)\cdots T(\Psi^{t-1}\xi),
\]
and assume that
\begin{enumerate}
\item[\textbf{A1:}] for all $\xi \in \Xi$, $T(\xi)\interior \R^d_+ \subset \interior \R^d_+$, and
\item[\textbf{A2:}] there exists $s\ge1$ such that, for all $\xi \in \Xi$, $T(\xi)\cdots T(\Psi^{s-1}\xi)(\R^d) \subset \{0\} \cup \interior \R^d$.
\end{enumerate}

Then there exist continuous maps $u, v : \Xi  \rightarrow \R^d_+$ with $\Vert u(\xi)\Vert =\Vert v(\xi)\Vert =1$ for all $\xi \in \Xi$ such that
\begin{itemize}
\item[(i)] the line bundles $E$ (resp. $F$) spanned by $u(\cdot)$ (resp. $v(\cdot)$) are such that $\R^{d} = E\bigoplus F^{\perp}$ where $b \in F(\xi)^{\perp}$ if and only if $\langle \xi, v(\xi)\rangle =0$.
\item[(ii)]  $E$ (resp. $F$) is $T,\Psi$-invariant (resp. $T^*,\Phi^{-1}$-invariant), i.e. $E(\Psi(\xi)) = E(\xi)T(\xi)$ and $F(\Psi \xi)T^*(\Psi\xi)=F(\xi)$, for all $\xi \in \Xi$;
\item[(iii)] there exist constants $\alpha <1$ and $C>0$ such that for all $t\ge 0$, and $\xi \in \Xi$,
\[
\Vert b T(\xi)\cdots T(\Psi^{t-1}\xi) \Vert \le C\alpha^t \Vert a T(\xi)\cdots T(\Psi^{t-1}\xi) \Vert,
\]
for all unit vectors $a \in E(\xi), b \in F(\xi)^{\perp}$.
\end{itemize}
\end{proposition}

Assumptions \textbf{H1-H2} imply that each continuous map $A_i : \Gamma  \rightarrow \mathbf{M}_{n_i}(\R)$ satisfies assumptions \textbf{A1-A2}. Hence Proposition \ref{prop:ruelle} applies to each continuous map $A_i$, and to the homeomorphism $\Theta$ on the compact space $\Gamma$. Then, for each of those maps, there exist row vector maps $u_i(\cdot)$, $v_i(\cdot)$, their respective vector bundles $E_i(\cdot)$, $F_i(\cdot)$, and the constant $C_i, \alpha_i >0$ satisfying properties (i), (ii), and (iii) of Proposition \ref{prop:ruelle}.

For each $i\in \{1,\dots,m\}$, define the continuous map $\bzeta_i : \Gamma \rightarrow \R$ by
\[
\bzeta_i(\gamma) := \ln \Vert u_i(\gamma)A_i(\gamma)\Vert.
\]

In the rest of this subsection, we deduce from Proposition \ref{prop:ruelle} some crucial properties of the invasions rates.

\begin{proposition}[Proposition 8.14 in \cite{roth-schreiber-14}]\label{prop_HS}
For all $\gamma \in \Gamma$ and every population $i$, $\rr_i(\gamma)$ satisfies the following properties:
\begin{enumerate}
\item[(i)] \[
\rr_i(\gamma)=\limsup_{t \to \infty}\frac{1}{t} \ln \Vert vA^t_i(\gamma)\Vert,
\]
for all $v \in \R^{n_i}_+\backslash \{0\}$.

\item[(ii)] \[
\rr_i(\gamma)=\limsup_{t \to \infty}\frac{1}{t} \sum_{s=0}^{t-1}\bzeta_i(\Theta^s(\gamma)).
\]
\end{enumerate}
\end{proposition}

\begin{proposition}\label{prop_intlambda}
The invasion rate of each population $i$ with respect to an $\Theta$-invariant measure $\tilde{\mu}$ satisfies the following property:
\[
\rr_i(\tilde{\mu}) = \int_{\Gamma} \bzeta_i(\gamma) \tilde{\mu}(d\gamma).
\]
\end{proposition}

\begin{proof}
This result is a direct consequence of property (ii) of Proposition \ref{prop_HS} and the Birkhoff's Ergodic Theorem applied to the continuous maps $\Theta$ and $\bzeta_i$.
\end{proof}

\begin{proposition}\label{cor_lambdanegative}
For all $(\bbx,c) \in \Gamma_+ \backslash \Gamma_0$, and every $i\in \{1,\dots,m\}$,
\[
\rr_i(\gamma) \le 0.
\]
\end{proposition}

\begin{proof}
Fix $i \in \{1,\dots,m\}$, and $(\bbx,c) \in \Gamma_+\backslash \Gamma_0$ with $x :=p \circ\pi_0(\bbx,c)$. In particular, $(\bbx,c)\in \Gamma_+^c\backslash \Gamma_0^c$ and  $x^i\in \R^{n_i}_+$ and $x^i \ne 0$. We have
\begin{eqnarray*}
x^iA^t_i(\bbx,c) &=& x^ic^t_i(x)\\
&=& (\Phi_c^t(x))^i,
\end{eqnarray*}
where the first equality is a consequence of (\ref{eq:cocycle}), and the second one follows from the definition of the map $\Phi_c$. Assumption \textbf{H3'} implies that there exists $T>0$ such that $\Phi^t_c(x)$ belongs to the compact set $V_c$ for all $t\ge T$, which implies that there exists $R>0$ such that $\Vert x^iA^t_i(\bbx,c) \Vert \le R$ for all $t\ge T$. Assertion (i) of Proposition \ref{prop_HS} applied to $v= x^i$ concludes the proof.
\end{proof}

\subsection{Properties of the empirical occupation measures}

Now we give some properties of the invasion rate with respect to a $\Theta$-invariant probability measure.

\begin{proposition}\label{prop_lambdanul}
For all $\tilde{\mu} \in \Inv(\Theta)(\Gamma_+)$ supported by $\Gamma_+ \backslash \Gamma_0$, $\rr_i(\tilde{\mu})=0$ for all $i \in \{1, \dots,m\}$.
\end{proposition}

\begin{proof}
Let $\tilde{\mu}$ be such a measure. Fix $i \in \{1, \dots,m\}$, and define the set $\Gamma^{i,{\eta}}:= \{ \gamma \in \Gamma_+ : \Vert \left(p\circ \pi_0(\gamma)\right)^i\Vert> \eta\}$, where $\left(p\circ \pi_0(\gamma)\right)^i \in \R_+^{n_i}$ is the $i$th sub-vector of $p\circ \pi_0(\gamma) \in \R_+^n$. By assumption on the measure $\tilde{\mu}$, there exists a real number $\eta^{*}>0$ such that $\tilde{\mu}(\Gamma^{i,{\eta}}) >0$ for all $\eta<\eta^*$.

The Poincar\'e  recurrence theorem applies to the map $\Theta$, and implies that for each $\eta<\eta^*$,
\begin{equation}\label{eq:poincare}
\tilde{\mu}(\{\gamma \in \Gamma^{i,{\eta}} \vert \ \Theta^t(\gamma) \in  \Gamma^{i,{\eta}} \text{ infinitely often }\})=1.
\end{equation}
Recall that the conjugacy (\ref{Eq_conj}) implies that for every $\gamma =(\bbx,c) \in \Gamma_+\backslash
\Gamma_0$ with $p\circ \pi_0(\gamma)=x \in V_c\backslash S_0$, we have
\begin{eqnarray*}
p\circ \pi_0(\Theta^t(\gamma))^i&=&\Phi_c^t(x)^i\\
&=&x^iA_i^t(\gamma).
\end{eqnarray*}
Then, equality (\ref{eq:poincare}) means that for $\tilde{\mu}$-almost all $\gamma \in \Gamma^{i,{\eta}}$ with $0<\eta <\eta^*$,  $ \Vert x^iA_i^t(\gamma) \Vert > \eta$ infinitely often. Therefore, Proposition \ref{prop_HS} (i), applied to $v=x^i$, implies that $\rr_i(\gamma)\ge0$ for $\tilde{\mu}$-almost all $\gamma \in \Gamma^{i,{\eta}}$, with  $\eta<\eta^*$. Hence, by Proposition \ref{cor_lambdanegative}, $\rr_i(\gamma)=0$ for $\tilde{\mu}$-almost all $\gamma \in  \bigcup_{n \ge \frac{1}{\eta^*}} \Gamma^{i,{1/n}} = \Gamma_+ \backslash \Gamma_0$, which completes the proof.
\end{proof}

\subsection{Proof of Proposition \ref{proposition-convergence-empiricalmeasures}}

\begin{lemma}\label{lemma:convergence-empiricalmeasures}
 For all $\gamma=(\bbx,c) \in\Gamma_+\backslash \Gamma_0$, every weak$^*$ limit point $\tilde{\mu}$ of the family of probability measures $\{\tilde{\Lambda}_t(\gamma)\}_{t\in \N}$ belongs to $\Inv(\Theta)(\Gamma_+^c)$ and satisfies $\rr_i(\tilde{\mu})\le 0$ for all $i$.
\end{lemma}

\begin{proof}
Let $\gamma=(\bbx,c) \in \Gamma_+\backslash \Gamma_0$. By Lemma \ref{prop_walters}, let $\tilde{\mu} = \lim_{k \to \infty}\tilde{\Lambda}_{t_k}(\gamma) \in \Inv(\Theta)(\Gamma_+^c)$. Proposition \ref{prop_intlambda}, the continuity of the maps $\bzeta_i$, and property (ii) of Proposition \ref{prop_HS}, imply the following equalities for all $i$:
\begin{eqnarray*}
\rr_i(\tilde{\mu})&=&\int_{\Gamma} \bzeta_i(\eta) \tilde{\mu}(d\eta) \\
&=& \lim_{k \to \infty}\frac{1}{t_k}\sum_{s=0}^{t_k-1}\bzeta_i(\Theta^s(\gamma))\\
&\le&\rr_i(\gamma).
\end{eqnarray*}
Hence, by Proposition \ref{cor_lambdanegative},
\[
\rr_i(\tilde{\mu}) \le 0, \ \text{ for all } i.
\]
\end{proof}

Now, we prove Proposition \ref{proposition-convergence-empiricalmeasures}. Let $x \in \R^n_+\backslash S_0$. By definition of the set $V_A$, there exists a time $T\ge 0$ such that $\Phi^t_A(x) \in V_A$, for all $t\ge T$. Choose $\gamma \in \pi_0^{-1}(\Phi^T_A(x),A) \subset \Gamma_+^A \backslash \Gamma_0^A$. Since $\mu$ is a weak$^*$ limit point of the family $\{\Lambda_t(\Phi^T_A(x))\}_{t\ge 0}$ if and only if it is a weak$^*$ limit point of the family $\{\Lambda_t(x)\}_{t\ge 0}$, we do not lose generality by considering $\{\Lambda_t(\Phi^T_A(x))\}_{t\ge 0}$. Since $V_A$ is compact, the set of all weak$^*$ limit points of the family of probability measures $\{\Lambda_t(\Phi^T_A(x))\}_{t\in \N}$ is a non-empty compact subset of $\mathcal{P}(V_A)$. Let $\mu = \lim_{k\to \infty} \Lambda_{t_k}(x)$ be such a weak$^*$ limit point. By Lemma \ref{lemma:convergence-empiricalmeasures}, passing to a subsequence if necessary, let $\tilde{\mu} = \lim_{k \to \infty}\tilde{\Lambda}_{t_k}(\gamma) \in \Inv(\Theta)(\Gamma_+^A)$ such that $\rr_i(\tilde{\mu})\le 0$ for all $i$. Furthermore by Lemma \ref{lemma_mes_emp} and continuity of $p \circ \pi_0$, $(p\circ \pi_{0})^*(\tilde{\mu})=\mu$. Hence Proposition \ref{prop:equality-lyap} and Lemma \ref{lemma:inv-measure} conclude the proof of the first assertion.

\subsection{Proof of Lemma \ref{lemma-inequality}}

Let $x \in S_0$. By the same arguments as in the above paragraph, there exist $T>0$, $\gamma \in \pi_0^{-1}(\Phi^T_A(x),A) \subset \Gamma_0^A$ and $\mu = \lim_{k\to \infty} \Lambda_{t_k}(x)$. Lemma~\ref{prop_walters} and the $\Theta$-invariance of $\Gamma_0^A$ imply that, passing to a subsequence if necessary, there exists $\tilde{\mu} = \lim_{k \to \infty}\tilde{\Lambda}_{t_k}(\gamma) \in \Inv(\Theta)(\Gamma_0^A)$ such that
\begin{eqnarray*}
\rr_i(\tilde{\mu})&=&\int_{\Gamma} \bzeta_i(\eta) \tilde{\mu}(d\eta) \\
&=& \lim_{k \to \infty}\frac{1}{t_k}\sum_{s=0}^{t_k-1}\bzeta_i(\Theta^s(\gamma))\\
&\le&\rr_i(\gamma).
\end{eqnarray*}
Note that since $\Phi^T_A(x)$ is on the trajectory of $x$, $\rr_i(x)=\rr_i(\Phi^T_A(x))$. Furthermore by Lemma \ref{lemma_mes_emp} and continuity of $p \circ \pi_0$, $(p\circ \pi_{0})^*(\tilde{\mu})=\mu$. Proposition~\ref{prop:equality-lyap} concludes the proof.
\subsection{Proof of Proposition \ref{proposition:robust-persistence}}

Let $\{\mu_n\}_{n\ge 1}$ be saturated invariant measures for the $\delta(A^n)$-perturbations
\[
X_{t+1}=X_tA^n(X_t),
\]
of model (\ref{DYN1}) defined at the beginning of the section. Assumption \textbf{H3'} implies that each measure $\mu_n$ is supported by the compact set $S_{A^n} \subset N_1(S_A)$. By weak* compactness of Borel probability measures on the compact set $N_1(S_A)$, there exist weak* limit points of $\{\mu_n\}_{n\ge0}$. Let $\mu \in \mathcal{P}(N_1(S_A))$ be such a weak* limit point. To ease the reading, when $A^n$ is on subscript or superscript, we write only $n$.

Now, we show that $\mu \in \Inv(\phi_A)$. Lemma \ref{lemma:inv-measure} implies that for all $n\ge 1$ there is a measure $\tilde{\mu}_n \in \Inv(\Theta)(\Gamma_+^{n})$ such that $(p\circ \pi_0)^*(\tilde{\mu}_n)=\mu_n$. By compactness of $\Inv(\Theta)(\Gamma_+)$, let $\tilde{\mu}:=\lim_{k\to \infty}\tilde{\mu}_{n_k} \in \Inv(\Theta)(\Gamma_+)$. The continuity of $p \circ \pi_0$ implies that $(p\circ \pi_0)^*(\tilde{\mu})=\mu$. We claim that $\support(\tilde{\mu}) \subset \Gamma_+^A$. For each $k\ge 1$, define
\[
B_k :=\bigcup_{l \ge k}\Gamma_+^{n_l} \cup \Gamma_+^A.
\]
Note that $B_{k+1} \subset B_k$ for all $k\ge 1$ and $\bigcap_k B_k = \Gamma_+^A$. Moreover $B_k$ is closed for all $k\ge1$. Indeed, the same arguments as in the proof of Lemma \ref{lemma:gammaclosed} can be used. Since $\support(\tilde{\mu_{n_k}})\subset \Gamma_+^{n_k}$ for all $k\ge 1$, $\lim_{m \to \infty}\tilde{\mu}_m(B_k)=1$ for all $k\ge 1$. The Portmanteau theorem (see e.g. Theorem 2.1 in \cite{billingsley-99}) implies that for all $k \ge 1$, $\tilde{\mu}(B_k)=1$. Then $\tilde{\mu}(\Gamma_+^A) = \tilde{\mu}(\bigcap_k B_k)=\lim_{k\to \infty}\tilde{\mu}(B_k) =1$. This proves the claim. Lemma~\ref{lemma:inv-measure} and the claim imply that $\mu \in \Inv(\Phi_A)(V_A)$

Next, we show that $\rr_i(\mu)\le 0$ for all $i$. Proposition \ref{prop_intlambda} and Proposition \ref{prop:equality-lyap} imply that, for all $n\ge0$ and all $i$,
\[
\rr_i(\mu_n) = \int_{\Gamma_+} \bzeta_i(\gamma)d\tilde{\mu}_n,
\]
with the convention that $\mu_0=\mu$ and $\tilde{\mu}_0=\tilde{\mu}$. On the other hand, by assumption, $\rr_i(\mu_n)\le 0$ for all $n\ge1$ and all $i$. Then, the continuity of $\bzeta_i$ and the convergence of $(\tilde{\mu}_n)_n$ to $\tilde{\mu}$ for the weak$^*$ topology imply that $\rr_i(\mu)\le0$.
$\; \; \blacksquare$

\section{Applications}\label{section:application}

\subsection{Lotka-Volterra difference equations} Even for models without population structure (i.e. $n_i=1$ for all $i$), our results extend results of \citep{garay-hofbauer-03} from homeomorphisms to non-invertible maps. An important class of these non-invertible maps are the discrete-time Lotka-Volterra equations introduced by \citep{hofbauer-etal-87} which are of the form
\begin{equation}\label{eq:LV}
x_{t+1} = x_t \circ \exp(B x_t + c)
\end{equation}
where $B$ is a $m\times m$ matrix, $c$ is a column vector of length $m$, $\exp(\cdot)$ denotes component-wise exponentiation, and $\circ$ denotes component-wise multiplication. A key feature of these equation is a time averaging property\red{~\cite[Lemma 2.4]{hofbauer-etal-87}}. Specifically, let  $I\subset\{1,\dots,m\}$ correspond to a subset of species and $K$ be a compact invariant set contained in \[
\interior \R_{+}^I:=\{x\in \R^m_+: x_i>0\mbox{ for all }i\in I\mbox{ and }x_i=0\mbox{ for all }i\notin I\},\] then there exists \red{a sequence $t_k\uparrow\infty$ and} an equilibrium $\hat x\in \interior \R_{+}^I$ such that
\[
\lim_{\red{k}\to\infty}\frac{1}{t\red{_k}} \sum_{s=1}^{\red{t_k}} x_s = \hat x
\]
whenever $x_0\in K$. In particular, this averaging property implies if $\mu$ is an ergodic measure with $\red{\mbox{species}(\mu)}=I$ \red{and there is a unique equilibrium $\hat x$ (which is generically true) in $\interior  \R_+^I$}, then
\[
\rr_i(\mu) =\sum_{j=1}^m B_{ij} \hat x_j +c_j
\]
for all $i$. Applying Theorem~\ref{thm:robust}, we get the following result.
\begin{theoreme}\label{thm:LV} Let $B$ be an $m\times m$ matrix and $c$ be a $m\times 1$ vector such that the Lotka-Volterra difference equation \eqref{eq:LV} is dissipative. Let $p_1,\dots,p_m$ be positive reals such that
\begin{equation}\label{eq:LVC}
\sum_{i=1}^m p_i \left(\sum_{j=1}^m B_{ij}\hat x_j +c_j\right)>0
\end{equation}
for every equilibrium $\hat x\in S_0$, then the Lotka-Volterra difference equation \eqref{eq:LV} is robustly permanent.
\end{theoreme}

The permanence condition \eqref{eq:LVC} is the same condition described by \citet[\red{Theorem 2.5}]{hofbauer-etal-87}. However, Theorem~\ref{thm:LV} implies the stronger result that permanence persists following sufficiently small perturbations of the right hand side of \eqref{eq:LV}. For example, consider the following model of competing annual plants with a seed bank
\begin{equation}\label{eq:annual}
x_{t+1}^i = g_i x_t^i \exp\left(Y_i -\sum_{j=1}^k C_{ij} g_j x_t^j\right)+(1-g_i)s_i x_t^i
\end{equation}
where $x_t^i$ is the density of seeds for species $i$, $g_i$ is the fraction of seeds of species $i$ germinating each year, $\exp(Y_i)$ is the yield of a germinating seed of species $i$ in the absence of competition,  $s_i\in [0,1)$ is the annual seed survivorship probability of species $i$,  and $C_{ij}>0$ is the competitive effect of germinated individuals of species $j$ on germinated individuals of species $i$.  Applying Theorem~\eqref{thm:LV} yields the following corollary.

\begin{corollary}
Let $C_{ij}>0$ and $Y_i>0$  for all $i,j$. If there exist $p_i>0$ such that
\[
\sum_{i=1}^m p_i \left(Y_i - \sum_{j=1}^m C_{ij}\hat x_j \right)>0
\]
\red{for every equilibrium $\hat x\in S_0$}, then there exists $\tilde g\in (0,1)$ such that the annual plant model \eqref{eq:annual} is permanent for all $g \in (\tilde g,1)^m$.
\end{corollary}

\begin{proof} To apply Theorem~\ref{thm:LV}, we need to verify that the annual plant model \eqref{eq:annual} is dissipative. Define $\tilde c=\min_i C_{ii}$, and $\tilde y = \max_i Y_i$. As  $z\exp(\tilde y - \tilde c z) \le \exp(\tilde y-1)/\tilde c=:a$ for all $z\ge 0$,  we have
\begin{equation}\label{eq:upper}
g_i x_i \exp \left( Y_i -\sum_{j=1}^k C_{ij} g_j x_j\right) \le  g_i x_i \exp \left( Y_i - C_{ii} g_i x_i\right)\le a
 \end{equation}
 for all $x\in [0,\infty)^m$. Define $b=\min_i (1-g_i) <1$.

 Let $x_t$ be a solution to the annual plant model \eqref{eq:annual} and let $\widetilde x_t$ be a solution to the linear difference equation \[
 \widetilde x_{t+1}=a (1,1,\dots,1)^T+b\widetilde x_t
 \]
 with $\widetilde x_0 =x_0$. We claim that $\widetilde x_t^i\ge x_t^i$ for all $t\ge 0$ and $i$.  By assumption, these inequalities hold for $t=0$. Now assume that they holds for some $t\ge 0$. Then inequality \eqref{eq:upper} and our choice of $b$ implies
 \[
 x_{t+1}^i \le a+bx_t^i \le a+ b \widetilde x_t^i = \widetilde x_{t+1}^i
 \]
 Since $\lim_{t\to\infty} \widetilde x_t^i = a/(1-b)<\infty$ for all $i$, it follows that the global attractor of the annual plant model \eqref{eq:annual} lies in the cube $[0,a/(1-b)]^k$ and the system is dissipative as claimed.

\end{proof}

\subsection{Competitve metacommunities} Metacommunities are ``populations of communities'' in which interacting species live in a finite number $k$ of patches coupled by dispersal. Here we consider a Lotka-Voltera metacommunities with two competing species. Let $X^{i}_t=(X^{i1}_t,\dots,X^{ik}_t)$ denote the row vector of populations abundances in the different patches for species $i$ at time $t\in \N$. Let $B^j$ be a $2\times 2$ matrix with positive entries representing the competition coefficients in patch $j$, and $c^j$ be a positive $2\times 1$ vector representing intrinsic rates of growth in patch $j$. For each species $i$, let $D^i=(d^i_{j\ell})_{\{1\le j,\ell\le k\}}$ be a column stochastic matrix whose $j,\ell$-th entry corresponds to the fraction of individuals of species $i$ in patch $\ell$ that disperse to patch $j$. Define the fitness of the species $i$ in patch $j$ as a function $f^i_j$ of the  density defined, for each $x\in \R_+^{2k}$, by
\[
f^i_j(x)=\exp(-\sum_{h}B^j_{ih}x^{hj}+c^j_i).
\]
Under these assumptions, the metacommunity dynamics are
\begin{equation}\label{eq:meta}
X^i_{t+1} = X^i_t \underbrace{ \rm{diag}\left\{f^i_1(X_t),\dots,f^i_k(X_t)\right\}D^i}_{A_i(X_t)}
\end{equation}

While Corollary~\ref{for:2species} provides a characterization of robust permanence for the dynamics of \eqref{eq:meta}, evaluating $\rr_i(\mu)$ for ergodic measures is, in general, challenging. However, using a perturbation result and the averaging property of Lotka-Volterra difference equations, we can prove the following result.

\begin{theoreme}\label{thm:spatial}
Assume $D^i$ are primitive and sufficiently close to the identity matrix i.e. $d^i_{jj}\approx 1$ for all $i,j$. Then the model \eqref{eq:meta} is robustly permanent if
\[
\max_j c^j_1 - B^j_{12}c^j_2/B^j_{22}>0 \mbox{ and }\max_j c^j_2 - B^j_{21}c^j_1/B^j_{11}>0,
\]
\red{and not permanent if one of the inequalities is reversed.}
\end{theoreme}

The condition for robust permanence correspond to each competitor being able to invade the equilibrium determined by the other competitor in some patch when the dynamics are not coupled by dispersal. This simple condition holds despite the dynamics of each competitor potentially being chaotic.

The proof of Theorem~\ref{thm:spatial} follows from Corollary~\ref{for:2species} and the following lower bound for the $\rr_i(\mu)$. To this end, define $S^i_0= \{x\in \R^{2k}_+ \ :\ \Vert x^l\Vert =0, l\neq i\}$ the set where only population $i$ is present.

\begin{proposition}\label{prop-approx}
For species $1$ and for any $\eps>0$, there exists $\delta_0>0$ such that if $D_{ii}^j \in (\delta_0,1]$ for all $i,j$, then
\[
\rr_2(\mu) \ge \max_j c^j_2 - B^j_{12}c^j_1/B^j_{11}-\eps
\]
for all invariant measures \red{$\mu$} of \eqref{eq:meta} supported by $S_0^{1}$.
\end{proposition}

\begin{proof}
To prove Proposition~\ref{prop-approx} , we need the following well-known lemma (see, for example, \cite[p.28]{bowen-75}).
\begin{lemma}\label{lemma-seq}
If $(a_t)_{t\ge0}$ is a sequence of real numbers such that $a_{t+s}\ge a_t + a_s$ for all $t,s \in \N$, then
\[
\lim_{t\to \infty}\frac{1}{t}a_t=\inf_{t\ge 0}\frac{1}{t}a_t.
\]
\end{lemma}

Fix $\eps>0$. Since $\lim_{x \to \infty}x\exp(-B^j_{11}x+c^j_1)=0$ for all $j=1,\dots,k$, there exists $N>0$ large enough such that the set
\[
\Gamma:=\{x \in S^1_0 \ :\ x^{1j}\le N, \forall j\}
\]
is \emph{one-step absorbing} for the dynamics of \eqref{eq:meta} restricted to $S^1_0$, i.e. for all $X^1_0 \in S^1_0$, $X^1_1 \in \Gamma$.  Denote $(Y_t)_{t\ge 0}$ the solution of
\[
\begin{aligned}
Y^1_{t+1} &= Y^1_t \underbrace{ \rm{diag}\left\{f^1_1(Y_t),\dots,f^1_k(Y_t)\right\}}_{=:B^1(Y_t)}  \quad Y_0\in S_0^1\\
Y^2_{t+1} &= Y^2_t \underbrace{ \rm{diag}\left\{f^2_1(Y_t),\dots,f^2_k(Y_t)\right\}}_{=:B^2(Y_t)} \\
\end{aligned}
\]
 i.e. the metapopulation dynamics for species $1$ without dispersal. Define $j= \argmax_{1\le l \le k}\{c^l_2 - B_{21}^lc^l_1/B^l_{11} \}$. The unique positive fixed point of the map $x \mapsto x\exp(-B^j_{11}x+c^j_1)$ is $c^j_1/B^j_{11}$. Hence the averaging property of Lotka-Volterra difference equation implies that
\begin{equation}\label{eq:conv}
 \lim_{t \to \infty} \frac{1}{t}\sum_{s=0}^{t-1}\log f^2_j(Y_s) = \left\{
   \begin{array}{ll}
       c^j_2 & \mbox{if } Y^{1j}_0=0 \\
       c^j_2-B_{21}^jc^j_1/B^j_{11}  & \mbox{otherwise.}
   \end{array}
\right.
\end{equation}

On the other hand, Theorem 1 in \cite{schreiber-98b} implies that
\begin{equation}\label{eq:conv-1}
\begin{aligned}
\inf_{x\in \Gamma}\left\{\lim_{t \to \infty} \frac{1}{t}\sum_{s=0}^{t-1}\log f^2_j(X_s):X_0=x\right\} &=\inf_{\mu \in \mathcal{M}_{\text{erg}}(\Gamma)}\int \log f^2_jd\mu \\
&= \sup_{t>0}\frac{1}{t}\inf_{x\in \Gamma}\left\{\sum_{s=0}^{t-1}\log f^2_j(X_s): X_0=x\right\},
\end{aligned}
\end{equation}
where $\mathcal{M}_{\text{erg}}(\Gamma)$ is the set of ergodic measures of the metapopulation dynamics \eqref{eq:meta} supported by $\Gamma$. Hence, Lemma~\ref{lemma-seq} and equalities \eqref{eq:conv} and \eqref{eq:conv-1} imply that there exists $T'>0$ such that, for all $t>T'$,
\begin{equation}\label{eq:conv-3}
\frac{1}{t}\sum_{s=0}^{t-1}\log f_j^2(Y_s) >c^j_2 - B^j_{21}c^j_1/B^j_{11} - \frac{\eps}{4},
\end{equation}
for all $Y_0\in\Gamma$
and there exists $T''>0$ such that
\begin{equation}\label{eq:conv-4}
\liminf_{t \to \infty} \frac{1}{t}\sum_{s=0}^{t-1}\log f^2_j(X_s)> \frac{1}{t}\sum_{s=0}^{t-1}\log f_j^2(X_s)- \frac{\eps}{4},
\end{equation}
for all $t>T''$ and $X_0 \in\Gamma$,

Define $T=\max\{T',T''\}$. By continuity of the functions $f^1_l, f^2_l$ for $l=1,\dots,k$, there exists $\eta>0$ such that,  whenever $|D_{jj}^i-1| \le \eta$ for all $i,j$,
\begin{equation}
 \frac{1}{T}\sum_{s=0}^{T-1}\log f_j^2(X_s)> \frac{1}{T}\sum_{s=0}^{T-1}\log f_j^2(Y_s) - \frac{\eps}{4}.
\end{equation}
for $Y_0=X_0\in \Gamma$.

Choose $\delta< \eta$ such that $\vert \log(1-\delta) \vert <\frac{\eps}{4}$. Then
\begin{eqnarray*}
r_2(x)&=& \limsup_{t\to \infty}\frac{1}{t}\log \Vert A_2(X_0)\dots A_2(X_{t-1}) \Vert\\
&\ge& \log(1-\delta) + \limsup_{t\to \infty}\frac{1}{t}\log \Vert B_2(X_0)\dots B_2(X_{t-1}) \Vert\\
&\ge&  \limsup_{t\to \infty}\frac{1}{t}\sum_{s=0}^{t-1} \log f_l^2(X_s)-\frac{\eps}{4} \\
&\ge& \liminf_{t\to \infty}\frac{1}{t}\sum_{s=0}^{t-1} \log f_l^2(X_s)-\frac{\eps}{4} \\
&\ge&\frac{1}{T}\sum_{s=0}^{T-1} \log f_l^2(X_s) - \frac{\eps}{2}\\
&\ge&  \frac{1}{T}\sum_{s=0}^{T-1} \log f^2_l(Y_s)-\frac{3\eps}{4}\\
&\ge&  c^j_2 - B^j_{21}c^j_1/B^j_{11}-\eps.
\end{eqnarray*}
for all $X_0=Y_0\in \Gamma$.
\end{proof}
\subsection{An SIR Epidemic Model~\label{application:disease}} A fundamental model in epidemiology is the SIR \red{model where} $S$ is the density of individuals susceptible to the disease, $I$ is the density of infected individuals, and $R$ is the density of removed individuals. If individuals die at a constant rate $\red{m}$ and encounter one another at the contact rate $\beta$, then a discrete-time version of this model is
\begin{equation}\label{eq:Model1}
\begin{array}{lcl}
S_{t+1} &=& f(S_t+I_t+R_t)(S_t+I_t+R_t)+e^{-\red{m}-\beta I_t}S_t\\
I_{t+1} &=& e^{-\red{m}} S_t(1-e^{-\beta I_t})\\
R_{t+1} &=& e^{-\red{m}} (I_t+R_t)
\end{array}
\end{equation}
where  $f$  is a continuous function describing reproduction that we assume satisfies
\begin{equation}\label{eq:Cond_f}
\limsup_{y\to \infty} f(y)<1-e^{-\red{m}}
\end{equation}
which guarantees the existence of the global compact trapping region $S_A$ in assumption \textbf{H3}.

Let $N=S+I+R$ be the total population size. Then \eqref{eq:Model1} is equivalent to
\begin{equation}\label{eq:Model2}
\begin{array}{lcl}
N_{t+1} &=& f(N_t)N_t+e^{-\red{m}}N_t\\
I_{t+1} &=& e^{-\red{m}} (N_t-I_t-R_t)(1-e^{-\beta I_t})\\
R_{t+1} &=& e^{-\red{m}} (I_t+R_t)
\end{array}
\end{equation}
with the state space $\{x\in \mathbb{R}^3_+\mid\: x_1\geq x_2+x_3\}$. Define
\[
u(y)=\left\{ \begin{array}{l}
\displaystyle \frac{1-e^{-\beta y}}{y},\mbox{ if }y>0\\
\\
\beta,\mbox{ if }y=0
\end{array}
\right.
\]
Then the SIR model \eqref{eq:Model2} is of the standard form \eqref{Model_Componentwise} with $m=2$, $X^1=N$, $X^2=(I,R)$,
\begin{equation}\label{eq:Matrices}
A_1(X)=f(N)+e^{-\red{m}},
\begin{array}{lcl}
A_2(X) &=&
\left(
\begin{array}{cc}
e^{-m}(N-I-R)u(I) & e^{-\red{m}}\\
0 & e^{-\red{m}}
\end{array}
\right)
\end{array}
\end{equation}
and extinction set $S_0=\{(N,I,R)\in \mathbb{R}^3_+\mid I=R=0\}$. The dynamics on $S_0$ are given by
\begin{equation}\label{eq:DF_Dynamics}
N_{t+1}=[f(N_t)+e^{-\red{m}}]N_t.
\end{equation}
While the matrix $A_2(X)$ is not irreducible, \red{we can use remark~\ref{remark} and decomposed it} into irreducible components, $A_2^1(X)=(e^{-\red{m}}(N-I-R)u(I))$ and $A_2^2(X)=(e^{-\red{m}})$.  \red{In particular, for any invariant probability measure $\mu$, we have
\begin{eqnarray*}
\rr_1(\mu)&=&\int \log(f(N)+e^{-m}) \,\mu(dNdIdR) \\
\rr_2(\mu)&=&\max\left\{ \int \log(e^{-m}(N-I-R)u(I)) \,\mu(dNdIdR),-m\right\}\\
\end{eqnarray*}
}
\red{Let us assume that the population persists in the absence of the disease i.e. $r_1(0)=\ln(f(0)+e^{-m})>0$. Theorem~\ref{thm:robust} implies that there is a positive global attractor $M_2\subset \R_+\times \{0\}\times \{0\}$ for the SIR dynamics restricted to $\R_+\times \{0\}\times \{0\}$. It follows that $\{M_1=\{(0,0,0)\},M_2\}$} is a Morse decomposition of $S_0$. \red{As $r_1(0)>0$ by assumption, Theorem \ref{thm:robust} implies robust permanence if $\int \log((N-I-R)u(I)) \,\mu(dNdIdR)>m$ for all invariant  probability measures $\mu$ with support in $M_2$.}  As a particular case, consider $f$ to be of the form $f(x)=\frac{1}{1+cx},\mbox{ for some constant }c>0.$ $f$ satisfies \eqref{eq:Cond_f} and $1+f'(x)>0$ for all $x\geq 0$. This implies that \red{forward orbits of} \eqref{eq:DF_Dynamics} \red{always converge to the globally stable equilibrium} $\bar{x}:=(1/c)(1/\sqrt{1-e^{-m}}-1)$. Thus, in this case, $M_2=\{\bar{x}\}$, and \eqref{eq:Model2} is robustly permanent if $(e^{-m \beta}/c)(1/\sqrt{1-e^{-m}}-1)>1$.
\vskip 0.1in
\paragraph{\red{{\bfseries Acknowledgements:} SJS was supported in part by US National Science Grant DMS\#1313418.} \blue{GR was supported by a start-up grant to SJS from the College of Biological Sciences, University of California, Davis and by the ERC Advanced Grant 322989 to Hal Caswell at the University of Amsterdam.}}

\bibliographystyle{plainnat}

\bibliography{seb-2}

\end{document}